\documentclass[a4paper 12pt]{amsart}
\usepackage[colorlinks]{hyperref}
\usepackage{xcolor}
\usepackage{mathscinet}
\usepackage{latexsym}
\usepackage{amsthm}
\usepackage{amssymb}
\usepackage{amsfonts}
\usepackage{amsmath}
\newtheorem{theorem}{Theorem}[section]
\newtheorem{thm}[theorem]{Theorem}

\newtheorem{lem}[theorem]{Lemma}

\newtheorem{prop}[theorem]{Proposition}

\newtheorem{assumption}[theorem]{Assumption}
\theoremstyle{definition}

\newtheorem{defn}[theorem]{Definition}
\theoremstyle{remark}

\newtheorem{rem}[theorem]{Remark}
\numberwithin{equation}{section}

 \DeclareMathAlphabet{\mathpzc}{OT1}{pzc}{m}{it}
 \newcommand{\To}{\longrightarrow}

 \newcommand{\proj}{\mathcal{P}}            
 \newcommand{\E}{\mathbb{E}}            
 \newcommand{\T}{\mathbb{T}}
 \newcommand{\e}{\varepsilon}
 \newcommand{\p}{\partial}
 \newcommand{\ek}{\mathcal{E}_K}
 \newcommand{\ekm}{\mathcal{E}_{m,K}}
 
 \newcommand{\D}{\mathcal{D}}
 \newcommand{\Ll}{\langle}
 \newcommand{\Rr}{\rangle}
 \newcommand{\ph}{\mathpzc{h}}
 \newcommand{\pl}{\mathpzc{l}}
 \newcommand{\N}{\mathbb{N}}
 \newcommand{\R}{\mathbb{R}}
 \newcommand{\Z}{\mathbb{Z}}
 \newcommand{\PP}{\mathbb{P}}
 \newcommand{\mcl}{\mathcal}

\begin{document}
\title[Exponential mixing of NSE with mildly degenerate noise]
{Exponential mixing of the 3D stochastic Navier-Stokes equations
driven by mildly degenerate noises}
\author[S.Albeverio]{Sergio Albeverio}
\address{Department of Applied Mathematics, Bonn University, Bonn, Germany}
\author[A. Debussche]{Arnaud Debussche}
\address{ENS Cachan, Antenne de Bretagne, Avenue Robert Schumann, 35170 BRUZ, France}
\email{arnaud.debussche@bretagne.ens-cachan.fr}
\author[L. Xu]{Lihu Xu}
\address{PO Box 513, EURANDOM, 5600 MB  Eindhoven. The Netherlands}
\email{xu@eurandom.tue.nl}
\date{}
\maketitle
\begin{abstract} \label{abstract}
\noindent We prove the strong Feller property and exponential mixing
for 3D stochastic Navier-Stokes equation driven by mildly degenerate
noises (i.e. all but finitely many Fourier modes are forced) via
Kolmogorov equation approach.
\end{abstract}

 \ \\
\section{Introduction}
The ergodicity of SPDEs driven by
\emph{degenerate} noises have been intensively studied in recent
years (see for instance \cite{EH01},\cite{HM06}, \cite{EM01}, \cite{HM09},
\cite{R08}). For the 2D stochastic Navier-Stokes equations (SNS),
there are several results on ergodicity, among which the most remarkable
one is by Hairer and
Mattingly (\cite{HM06}). They proved that the 2D stochastic dynamics has
a unique invariant measure as long as the noise forces at least two
linearly independent Fourier modes. As for the 3D SNS, most of ergodicity
results are about the dynamics driven by non-degenerate noises (see \cite{DPD03},
\cite{FR08}, \cite{R04}, \cite{R08}, \cite{RZ08}). In the respect of the degenerate noise case,
as noises are essentially elliptic setting of which \emph{all but finite} Fourier modes are driven,
\cite{RX09} obtained the ergodicity by combining Markov selection
and Malliavin calculus. As the noises are
truly hypoelliptic (\cite{HM06}), ergodicity is still open.
\ \\

In this paper, we shall still study the 3D SNS driven by essentially elliptic noises as above, but our approach is essentially different from that in \cite{RX09}.
Rather than Markov selection and cutoff technique, we prove the strong Feller property
by studying some Kolmogorov equations with a large negative potential, which was developed in
\cite{DPD03}.
Comparing with the method in \cite{DPD03} and \cite{DO06},
we cannot apply the Bismut-Elworthy-Li formula (\cite{EL94}) due to the degeneracy of the noises.
To fix this problem, we follow the ideas in \cite{EH01} and split the dynamics into high and low frequency parts, applying
the formula to the dynamics at high modes and Malliavin calculus to those at low ones.
Due to the degeneracy of the noises again, when applying Duhamel formula as in \cite{DPD03} and \cite{DO06},
we shall encounter an obstruction of not integrability (see \eqref{e:WroDuh}). Two techniques are developed in Proposition
\ref{p:LesGam} and \ref{p:HigAx} to conquer this problem, and the underlying idea is to trade off the spatial regularity
for the time integrability. Using the coupling method of \cite{Od07}, in which the noises have to be non-degenerate, we prove the exponential mixing and find that the construction of the coupling can be simplified.
Finally, we remark that the large coefficient $K$ in front of the potential
(see \eqref{e:KEMFK}), besides suppressing the nonlinearity $B(u,u)$ as in \cite{DPD03} and \cite{DO06},
also conquers the crossing derivative flows (see \eqref{e:DhHXL} and \eqref{e:DhLXH}).
\ \\

Let us discuss the further application of the Kolmogorov equation method in \cite{DPD03}, \cite{DO06} and this paper. For another essentially elliptic setting where sufficiently large (but still finite) modes are forced (\cite{HM06}, section 4.5), due to the large negative potential, it is easy to show the \emph{asymptotic strong Feller} (\cite{HM06}) for the semigroup $S^m_t$ (see \eqref{e:StmDef}). There is a hope to transfer this asymptotic strong Feller to the semigroup $P^m_t$ (see \eqref{e:PtmDef}) using the technique in Proposition \ref{p:HigAx}. If $P^m_t$ satisfies asymptotic strong Feller, then we can also prove the ergodicity. This is the further aim of our future research in 3D SNS. \\

The paper is organized as follows. Section \ref{s:PreMai} gives
a detailed description of the problem, the assumptions on the
noise and the main results (Theorems \ref{t:MaiThm1}
and \ref{t:MaiThm2}). Section \ref{s:KeyLem} proves the crucial estimate in Theorem
\ref{thm:StBound}, while the fourth section
\ref{s:MalCal} applies Malliavin calculus
to prove the important Lemma \ref{lem:LowDerivative}.
Section \ref{s:ProMaiThm} gives a sketch proof for the main theorems, and
the last section contains the estimate of Malliavin matrices and the proof of some
technical lemmas. \\

{\bf Acknowledgements}: We would like to thank Prof. Martin Hairer
for pointing out a serious error in the original version and some helpful suggestions on how to correct it.
We also would like to thank Dr. Marco Romito for the
stimulating discussions and some helpful suggestions on correcting
several errors.
\section{Preliminary and main results} \label{s:PreMai}
\subsection{Notations and assumptions}
Let ${\T^3}=[0,2\pi]^3$ be the three-dimensional torus, let
\begin{equation*} \label{e:CalH}
H=\{x \in L^2(\T^3,{\R}^3): \int_{{\T}^3} x(\xi)d\xi=0, \ div x(\xi)=0\},
\end{equation*}
and let
$$\mathcal{P}:L^2(\T^3, \R^3) \rightarrow H$$
be the orthogonal projection operator.
We shall study the equation
\begin{equation} \label{e:NSEabs}
\begin{cases}
dX + [\nu A X + B(X,X)] dt = QdW_t,\\
X(0) =x,
\end{cases}
\end{equation}
where
\begin{itemize}
\item $A = -\mathcal{P} \Delta$ \ \ $D(A)=H^2(\T^3,{\R}^3) \cap H$.
\item The nonlinear term $B$ is defined by
\begin{equation*} \label{e:NonlinearB}
 B(u,v)=\mathcal{P}[(u \cdot\nabla)v], \ B(u)=B(u,u) \ \ \ \forall \ u, v \in H^1(\T^3,{\R}^3) \cap H.
\end{equation*}
\item $W_t$ is the cylindrical Brownian motion on $H$ and $Q$ is the covariance matrix to be defined later.
\item We shall assume the value $\nu=1$ later on, as its
exact value will play no essential role.
\end{itemize}
\  \\

Define ${\Z}^3_{+}=\{k \in {\Z}^3; k_1>0 \} \cup \{k \in {\Z}^3; k_1=0, k_2>0 \} \cup \{k \in {\Z}^3; k_1=0, k_2=0,
k_3>0\},$ ${\Z}^3_{-}=-{\Z}^3_{+}$ and ${\Z}_{*}^3={\Z}^3_{+} \cup {\Z}^3_{-}$, for any $n>0$, denote
\begin{equation*} \label{e:ZlnZhn}
Z_\pl(n)=[-n,n]^3 \setminus (0,0,0), \ \ \ Z_\ph(n)={\Z}_{*}^3 \setminus Z_\pl(n).
\end{equation*}
Let $k^{\bot}=\{\eta \in {\R}^3; k \cdot \eta=0\}$,
define the projection $\mathcal{P}_k: {\R}^3 \rightarrow k^{\bot}$ by
\begin{equation} \label{e:ProK}
\mathcal{P}_k \eta=\eta-\frac{k \cdot \eta}{|k|^2} k \ \ \ \ \ \ \eta \in {\R}^3.
\end{equation}
Let $e_k(\xi)=cos k\xi \ {\rm if} \ k \in {\Z}^3_{+}$, $\ e_k(\xi)=sin k\xi \ {\rm if} \ k \in {\Z}^3_{-}$ and let $\{e_{k,1},e_{k,2}\}$ be an orthonormal basis of $k^{\bot}$, denote
$$e^1_k(\xi)=e_k(\xi) e_{k,1}, \ \ e^2_k(\xi)=e_k(\xi) e_{k,2}  \ \ \ \ \forall \ k \in {\Z}^3_*;$$
$\{e^i_k;k \in {\Z}^3_{*},i=1,2\}$ is a Fourier basis
of $H$ (up to the constant $\sqrt{2}/(2 \pi)^{3/2}$). With this
Fourier basis, we can write the cylindrical Brownian motion $W$ on $H$ by
$$W_t=\sum_{k \in {\Z}_{*}^3} w_k(t)e_k=\sum_{k \in {\Z}_{*}^3} \sum_{i=1}^2 w^i_k(t)e^i_k$$
where each $w_k(t)=(w^1_k(t),w^2_k(t))^T$ is a 2-d standard Brownian motion.
Moreover,
$$B(u,v)=\sum_{k \in {\Z}_{*}^3} B_k(u,v) e_k$$
where $B_k(u,v)$ is the Fourier coefficient of $B(u,v)$ at the mode $k$. Define
$$\tilde B(u,v)=B(u,v)+B(v,u), \ \ \ \ \tilde B_k(u,v)=B_k(u,v)+B_k(v,u).$$
We shall calculate $\tilde B_k(a_j e_j, a_l e_l)$ with $a_j \in j^{\bot}, a_l \in l^{\bot}$ in Appendix \ref{sub:App1}.
\ \\

Furthermore, given any $n>0$, let $\pi_n: H \To H$ be the projection
from $H$ to the subspace $\pi_n H:=\{x \in H: x=\sum_{k \in Z_\pl (n)} x_k e_k\}$.

\begin{assumption}[Assumptions for $Q$] \label{a:Q}

We assume that $Q:H \To H$ is a linear bounded operator such that
\begin{itemize}
\item[(A1)] (Diagonality) There are a sequence of linear maps $\{q_k\}_{k \in {\Z}_{*}^3}$ with $q_k: k^\bot \To k^\bot$ such that
$$Q(y e_k)=(q_k y) e_k \ \ \ \  \ y \in k^\bot.$$
\item[(A2)] (Finitely Degeneracy) There exists some nonempty sublattice $Z_\pl(n_0)$ of ${\Z}_{*}^3$ such that
$$q_k=0 \ \ \ \ \ \ k \in Z_\pl(n_0).$$
\item[(A3)] $(Id-\pi_{n_0}) A^{r} Q$ is \emph{bounded invertible} on $(Id-\pi_{n_0})H$ with $1<r<3/2$ and moreover $Tr[A^{1+\sigma} QQ^{*}]<\infty$ for some $\sigma>0$.
\end{itemize}
\end{assumption}
\begin{rem}  \label{rem:Q-matrix}
Under the Fourier basis of $H$, $Q$ has the following representation
\begin{equation} \label{e:QRep}
Q=\sum_{k \in Z_{\ph}(n_0)} \sum_{i,j=1}^2 q^{ij}_k e^i_k \otimes e^j_k
\end{equation}
where $x \otimes y: H \To H$ is defined by $(x \otimes y)z=\Ll y,z
\Rr x$ and $(q^{ij}_k)$ is a matrix representation of $q_k$ under
some orthonormal basis $(e_{k,1},e_{k,2})$ of $k^\bot$. By (A3), $rank(q_k)=2$ for all $k \in
Z_\ph(n_0)$. Take $Q=(Id-\pi_{n_0})A^{-r}$ with some $5/4<r<3/2$,
it clearly satisfies (A1)-(A3).
\end{rem}

With the above notations and assumptions, equation \eqref{e:NSEabs} can be represented under the Fourier basis by
\begin{equation} \label{e:NSEFourier}
\begin{cases}
dX_{k}+[|k|^2 X_{k}+B_{k}(X)]dt
=q_k dw_k(t), \ \ k \in Z_\ph (n_0)  \\
dX_{k}+[|k|^2 X_{k}+B_{k}(X)]dt=0, \ \ k \in Z_\pl(n_0)  \\
X_{k}(0)=x_{k}, \ \ \ k \in {\Z}_{*}^3
\end{cases}
\end{equation}
where $x_k, X_k, B_k(X) \in k^\bot$.
\ \\

We further need the following notations:
\begin{itemize}
\item $\mathcal B_b(B)$ denotes the Borel measurable bounded function space on the given Banach space $B$. $|\cdot|_{B}$ denotes the norm of a given Banach space $B$
\item $|\cdot|$ and $\Ll \cdot, \cdot \Rr$ denote the norm and the inner product of $H$ respectively.
\item Given any $\phi \in C(D(A),\R)$, we denote
\begin{equation} \label{e:Dhphi}
D_h \phi(x):=\lim_{\epsilon \rightarrow 0} \frac{\phi(x+\e h)-\phi(x)}{\e},
\end{equation}
provided the above limit exists, it is natural to define $D \phi(x): D(A) \rightarrow \R$ by $D \phi(x) h=D_h \phi(x)$ for all $h \in D(A)$. Clearly, $D \phi(x) \in D(A^{-1})$. We call $D \phi$ the first order derivative of $\phi$, similarly, one can define
the second order derivative $D^2 \phi$ and so on. Denote $C^k_b(D(A),\mathbb R)$ the set of functions from $D(A)$ to $\mathbb{R}$ with bounded
$0$-th, $\ldots$, $k$-th order derivatives.
\item Let $B$ be some Banach space and $k \in \mathbb{Z}_{+}$, define $C_k(D(A),B)$ as the function space from $D(A)$ to $B$ with the norm
\begin{equation*} \label{e:SmoPhi}
||\phi||_k:=\sup_{x \in D(A)} \frac{|\phi(x)|_B}{1+|Ax|^k} \ \ \ \phi \in C_k(D(A),B).
\end{equation*}
\item For any $\gamma>0$ and $0 \leq \beta \leq 1$, define the H$\ddot{o}$lder's norm $||\cdot||_{2,\beta}$ by
$$||\phi||_{2,\beta}=\sup_{x, y\in D(A)} \frac{|\phi(x)-\phi(y)|}{|A^{\gamma}(x-y)|^\beta (1+|Ax|^2+|Ay|^2)},$$
and the function space $C^{\beta}_{2,\gamma} (D(A),\R)$ by
\begin{equation} \label{e:HolSpa}
 C^{\beta}_{2,\gamma}(D(A),\R)=\{\phi \in C_2(D(A), \R); ||\phi||_{C^\beta_{2,\gamma}}= ||\phi||_2+||\phi||_{2,\beta}<\infty\}.
\end{equation}
\end{itemize}
\ \

\subsection{Main results}
The following definition of Markov family follows that in \cite{DO06}.
 \begin{defn} \label{d:MarFam}
 Let $(\Omega_x, \mathcal F_x, \PP_x)_{x \in D(A)}$ be a family of probability spaces
 and let $(X(\cdot,x))_{x \in D(A)}$ be a family of stochastic processes on
 $(\Omega_x, \mathcal F_x, \PP_x)_{x \in D(A)}$. Let $(\mathcal F^t_x)_{t \geq 0}$
 be the filtration generated by $X(\cdot,x)$ and let $\mathcal P_x$ be the law of
 $X(\cdot,x)$ under $\PP_x$. The family of
 $(\Omega_x, \mathcal F_x, \PP_x, X(\cdot,x))_{x \in D(A)}$ is a Markov family if the following condition
 hold:
 \begin{enumerate}
 \item For any $x \in D(A)$, $t \geq 0$, we have
 \begin{equation*}
 \PP_x(X(t,x) \in D(A))=1.
 \end{equation*}
 \item The map $x \rightarrow \mathcal P_x$ is measurable. For any $x \in D(A)$, $t_0,
 \cdots, t_n \geq 0$ , $A_0, \cdots, A_n \subset D(A)$ Borel measurable, we have
  \begin{equation*}
  \PP_x(X(t+\cdot) \in \mathcal A|\mathcal F^t_x)=\mathcal{P}_{X(t,x)}(\mathcal A)
  \end{equation*}
  where $\mathcal A=\{(y(t_0), \cdots, y(t_n)); y(t_0) \in A_0, \cdots, y(t_n) \in A_n \}$.
 \end{enumerate}
 The Markov transition semigroup $(P_t)_{t \geq 0}$ associated to the family is then defined
 by
 \begin{equation*}
 P_t \phi(x)=\E_x[\phi(X(t,x))],  \ \ x \in D(A) \ \ t \geq 0.
 \end{equation*}
 for all $\phi \in \mathcal B_b(D(A), \R)$.
 \end{defn}
 The main theorems of this paper are as the following, and will be proven in Section \ref{s:ProMaiThm}. \begin{thm}  \label{t:MaiThm1}
 There exists a Markov family of martingale solution
 $(\Omega_x, \mathcal F_x, \PP_x, X(\cdot,x))_{x \in D(A)}$
 of the equation \eqref{e:NSEabs}. Furthermore, the transition semigroup
 $(P_t)_{t \geq 0}$ is stochastically continuous.
 \end{thm}
 \begin{thm} \label{t:MaiThm2}
  The transition semigroup $(P_t)_{t \geq 0}$ in the previous theorem is strong Feller and irreducible. Moreover, it admits a unique invariant measure $\nu$ supported on $D(A)$ such that, for any probability measure $\mu$ supported on $D(A)$,
 we have
 \begin{equation} \label{e:ExpMix}
||P^{*}_t \mu-\nu||_{var} \leq Ce^{-ct}\left(1+\int_H |x|^2 \mu(dx)\right)
 \end{equation}
where $||\cdot||_{var}$ is the total variation of signed measures, and $C,c>0$ are the constants depending on $Q$.
 \end{thm}
\subsection{Kolmogorov equations for Galerkin approximation}
Let us consider the Galerkin approximations of the equation \eqref{e:NSEFourier} as follows
\begin{equation} \label{e:GalerkinM}
\begin{cases}
dX_m=-[AX_m+B_m(X_m)]dt+Q_m dW_t \\
X_m(0)=x_m
\end{cases}
\end{equation}
where $x_m \in \pi_m D(A)$, $B_m(x)=\pi_m B(\pi_m x)$ and $Q_m=\pi_m
Q$. The Kolmogorov equation for
\eqref{e:GalerkinM} is
\begin{equation} \label{e:KEM}
\begin{cases}
\partial_t u_m=\frac 12 Tr[Q_m Q_m^{*} D^2 u_m]-\Ll A x+B_m(x), Du_m \Rr  \\
u_m(0)=\phi
\end{cases}
\end{equation}
where $\phi$ is some suitable test function and
$$\mathcal{L}_m:=\frac 12 Tr[Q_m Q_m^{*} D^2 ]-\Ll A x+B_m(x), D
\Rr$$ is the Kolmogorov operator associated to \eqref{e:GalerkinM}. It is well known that
\eqref{e:KEM} is uniquely solved by
\begin{equation}
u_m(t,x)=\E[\phi(X_m(t,x))], \ \ \ \ x \in \pi_m D(A).
\end{equation}
Now we introduce an auxiliary Kolmogorov equation with a
 negative potential $-K|Ax|^2 $ as
\begin{equation} \label{e:KEMFK}
\begin{cases}
\partial_t v_m=\frac 12 Tr[Q_m Q_m^{*} D^2 v_m]-\Ll A x+B_m(x), Dv_m \Rr-K |Ax|^2 v_m,  \\
v_m(0)=\phi,
\end{cases}
\end{equation}
which is solved by the following Feynman-Kac formula
\begin{equation}
v_m(t,x)=\E \left[\phi(X_m(t,x))\exp \{-K \int_0^t |AX_m(s,x)|^2ds\}\right].
\end{equation}
Denote
\begin{equation} \label{e:StmDef}
S^m_t \phi(x)=v_m(t,x),
\end{equation}
\begin{equation} \label{e:PtmDef}
P^m_t \phi(x)=u_m(t,x),
\end{equation}
for any $\phi \in \mathcal{B}(\pi_m D(A))$, it is clear that $S^m_t$ and $P^m_t$ are both contraction semigroups on
$\mathcal{B}(\pi_m D(A))$. By Duhamel's formula, we have
\begin{equation} \label{e:DuhUmt}
u_m(t)=S^m_t \phi+K\int_0^t S^m_{t-s}[|Ax|^2 u_m(s)]ds.
\end{equation}
For further use, denote
\begin{equation} \label{e:ekm}
\ekm(t)=\exp \{-K \int_0^t |AX_m(s)|^2ds\},
\end{equation}
which plays a very important role in section \ref{s:KeyLem}. The $K>1$ in
\eqref{e:ekm} is a large but \emph{fixed} number, which conquers the crossing derivative flows (see \eqref{e:DhHXL} and \eqref{e:DhLXH}).
We will often use the trivial fact $\mathcal{E}_{m,K_1+K_2}(t)=\mathcal{E}_{m,K_1}(t)\mathcal{E}_{m,K_2}(t)$ and
\begin{equation} \label{e:IntEmk}
\int_0^t |AX_m(s)|^2 \ekm(s) ds=\frac{1}{K} (1-\ekm(t)) \leq \frac 1K.
\end{equation}

\section{Gradient estimate for the semigroups $S^m_t$} \label{s:KeyLem}
In this section, the main result is as follows, and it is similar to Lemma 3.4 in \cite{DO06} (or Lemma
4.8 in \cite{DPD03}).
\begin{thm} \label{thm:StBound}
 Given any $T>0$ and $k \in \mathbb{Z}_+$, there exists some $p>1$ such that for any $\max\{\frac 12, r-\frac 12\}<\gamma \leq 1$
 with $\gamma \neq 3/4$ and $r$ defined in Assumption \ref{a:Q}, we have
\begin{equation} \label{e:StEst}
||A^{-\gamma}DS^m_t \phi||_k \leq C t^{-\alpha} ||\phi||_k \ \ \ \
0<t \leq T
\end{equation}
for all $\phi \in C^1_b(D(A), \mathbb R)$, where $C=C(k,
\gamma,r,T,K)>0$ and $\alpha=p+\frac 12+r-\gamma$.
\end{thm}
\begin{rem}
The condition '$\gamma \neq 3/4$' is
due to the estimate \eqref{e:NonLinEst} about the nonlinearity
$B(u,v)$.
\end{rem}
\cite{DPD03} and \cite{DO06} proved the estimate \eqref{e:StEst} by applying the identity  
\begin{equation}   \label{e:BELFK}
\begin{split}
D_h S^m_t \phi(x)&=\frac{1}{t} \E [\mathcal{E}_{m,K} (t) \phi(X^m(t,x))  \int_0^t \Ll Q^{-1} D_h X^m(s,x), dW_s \Rr] \\
&+2K \E \left[\ekm(t) \phi(X^m(t,x)) \int_0^t (1-\frac{s}{t}) \Ll AX^m(s,x), AD_hX^m(s,x) \Rr ds \right],
\end{split}
\end{equation}
and bounding the two terms on the r.h.s. of \eqref{e:BELFK}.
 Since the $Q$ in Assumption \ref{a:Q} is \emph{degenerate}, the
formula \eqref{e:BELFK} is not available in our case. Alternatively,
we apply the idea in \cite{EH01} to fix this problem, i.e. splitting $X_m(t)$ into the \emph{low} and \emph{high} frequency parts,
and applying Malliavin calculus and Bismut-Elworthy-Li formula on the them respectively.
\  \\

Let $n \in \mathbb N$  be a \emph{fixed} number throughout this paper which satisfies $n>n_0$ and
will be determined later ($n_0$ is the constant in Assumption \ref{a:Q}). We split the Hilbert space $H$ into
the low and high frequency parts by
\begin{equation} \label{e:LowHigH}
\pi^\pl H=\pi_n H, \ \ \ \ \pi^\ph H=(Id-\pi_n) H.
\end{equation}
(We remark that the technique of splitting frequency space into two pieces is similar to the well
known Littlewood-Paley projection in Fourier analysis.) Then, the Galerkin approximation \eqref{e:GalerkinM}
with \underline{$m>n$} can be divided into two parts as follows:
\begin{equation} \label{e:LHEquation}
\begin{split}
&dX_m^{\pl}+[A X_m^{\pl}+B_m^{\pl}(X_m)]dt=Q_m^{\pl}dW^\pl_t  \\
&dX_m^{\ph}+[A X_m^{\ph}+B_m^{\ph}(X_m)]dt=Q_m^{\ph}dW^\ph_t
\end{split}
\end{equation}
where $X_m^{\pl}=\pi^\pl X_m, \ X_m^{\ph}=\pi^\ph X_m$ and the other terms are defined in the same way. In particular,
\begin{equation} \label{e:QlQhRep}
Q_m^{\pl}=\sum_{k \in Z_\pl(n) \setminus Z_\pl(n_0)} \sum_{i,j=1}^2 q^{ij}_k e^i_k \otimes e^j_k \ ,
\ \ \ Q_m^{\ph}=\sum_{k \in Z_\pl(m) \setminus Z_\pl(n)} \sum_{i,j=1}^2 q^{ij}_k e^i_k \otimes e^j_k,
\end{equation}
with $x \otimes y: H \To H$ defined by $(x \otimes y)z=\Ll y,z \Rr x$
\ \\

With such separation for the dynamics, it is natural to split the Frechet derivatives on $H$ into the low and high frequency parts.
More precisely, for any stochastic process $\Phi(t,x)$ on $H$ with $\Phi(0,x)=x$,
the Frechet derivative $D_h\Phi(t,x)$ is defined by
\begin{equation*}
D_h \Phi(t,x):=\lim \limits_{\epsilon \rightarrow 0} \frac{\Phi(t,x+\epsilon
h)-\Phi(t, x)}{\epsilon} \ \  \ \ \ \ h \in H,
\end{equation*}
provided the limit exists. The map $D\Phi(t,x): H \To H$ is naturally defined by
$D\Phi(t,x)h=D_h \Phi(t,x)$ for all $h \in H$.
Similarly, one can easily define $D^\pl \Phi(t,x)$,
$D^\ph \Phi(t,x)$, $D^\pl \Phi^\ph(t,x)$, $D^\ph \Phi^\pl(t,x)$ and so on, for instance,
$D^\ph \Phi^\pl(t,x): \pi^\ph H \rightarrow \pi^\pl H$
is defined by
$$D^\ph \Phi^\pl(t,x) h=D_{h}\Phi^\pl(t,x) \ \ \ \forall \ h \in \pi^\ph H$$
with $D_{h}\Phi^\pl(t, x)=\lim \limits_{\epsilon \rightarrow 0} [\Phi^\pl(t,x+\epsilon h)-\Phi^\pl(t,x)]/\epsilon$.

Recall that for any $\phi \in C^1_b(D(A),\R)$ one can define $D \phi$ by
\eqref{e:Dhphi}, in a similar way as above, $D^\pl \phi(x)$ and $D^\ph \phi(x)$) can be defined (e.g.
$D^\pl \phi(x) h=\lim_{\e \rightarrow 0}[\phi(x+\e h)-\phi(x)]/\e \ \ \ h \in D(A)^\pl$).
\ \\
\begin{lem} \label{l:XmEmKEst}
Denote $Z(t)=\int_0^t e^{-A(t-s)}Q dW_t$, for any $T>0$ and $\e<\sigma/2$
with the $\sigma$ as
in Assumption \ref{a:Q}, one has
\begin{equation} \label{e:ZtEst}
\E \left[\sup \limits_{0 \leq t \leq T} |A^{1+\e} Z(t)|^{2k}\right] \leq C(\alpha) T^{2k(\sigma-2\e-2\alpha)}
\end{equation}
where $0<\alpha<\sigma /2-\e$ and $k \in \mathbb{Z}_+$. Moreover, as $K>0$ 
is sufficiently large, for any $T>0$ and any $k \geq 2$, we have
\begin{equation} \label{e:XmEmKEst}
\E \left[\sup_{0 \leq t \leq T} \ekm(t) |AX_m(t)|^k \right] \leq C(k,T) (1+|Ax|^k).
\end{equation}
\end{lem}
\begin{proof}
The proof of \eqref{e:ZtEst} is standard (see Proposition 3.1 of \cite{DO06}).
Writing
$X_m(t)=Y_m(t)+Z_m(t)$,
and differentiating $|AY_m(t)|^2$ (or seeing (3.1) in Lemma 3.1 of \cite{DPD03}), we have
\begin{equation*}
\mathcal{E}_{m,K}(t) |AY_m(t)|^2 \leq |Ax|^2+\sup \limits_{0 \leq t \leq T} |AZ_m(t)|^2.
\end{equation*}
as $K>0$ is sufficiently large. Hence,
$$\ekm(t) |AX_m(t)|^2 \leq \ekm(t) |AY_m(t)|^2+\ekm(t) |AZ_m(t)|^2 \leq |Ax|^2+2\sup \limits_{0 \leq t \leq T} |AZ(t)|^2.$$
Hence, by \eqref{e:ZtEst} and the above inequality, we immediately have \eqref{e:XmEmKEst}. 
\end{proof}
The main ingredients of the proof of Theorem \ref{thm:StBound} are the following two lemmas,
and they will be proven in Appendix \ref{app:TecLem} and Section \ref{sub:LowDerivative} respectively.

\begin{lem} \label{l:HighDerivative}
Let $x \in D(A)$ and let $X_m(t)$ be the solution to \eqref{e:GalerkinM}.
Then, for any $\max\{\frac 12, r-\frac 12\}<\gamma \leq 1$ with $\gamma \neq 3/4$,
$h \in \pi_m H$ and $v \in L_{loc}^2(0, \infty; H)$,
as $K$ is sufficiently large, we have almost surely
\begin{align}
|A^{\gamma}D_hX_m(t)|^2 \ekm(t)+\int_0^t |A^{1/2+\gamma} D_hX_m(s)|^2 \ekm(s) ds
\leq |A^\gamma h|^2  \label{e:DhX} \\
|A^{\gamma}D_{h^\ph}X_m^{\pl}(t)|^2 \ekm(t) \leq \frac{C} K |A^{\gamma} h|^2  \label{e:DhHXL}\\
|A^{\gamma}D_{h^\pl}X_m^{\ph}(t)|^2 \ekm(t) \leq \frac{C} K |A^\gamma h|^2  \label{e:DhLXH} \\
\int_0^t |A^r D_hX_m(s)|^2 \ekm(s) ds \leq C t^{1-2(r-\gamma)} |A^\gamma h|^{2} \label{e:ArDhX} \\
\E[\ekm(t)\int_0^t \Ll v(s),dW(s) \Rr] \leq \E[\int_0^t \ekm^2(s)|v(s)|^2 ds] \label{e:EkMat}
\end{align}
where all the $C=C(\gamma)>0$ above are independent of $m$ and $K$.
\end{lem}
\begin{lem} \label{lem:LowDerivative}
Given any $\phi \in C^1_b(D(A))$ and $h \in \pi^\pl H$, there exists some $p>1$ (possibly very large) such that for any $k \in \mathbb{Z}_+$, we have some constant
$C=C(p,k)>0$ such that
\begin{equation} \label{e:LowDerivative}
|\E[D^\pl \phi(X_m(t))D_h X_m^{\pl}(t,x) \ekm(t)]| \leq C t^{-p} e^{Ct}||\phi||_{k} (1+|Ax|^k) |h|
\end{equation}
\end{lem}
\ \ \ \
\begin{proof} [{Proof of Theorem \ref{thm:StBound}}]
For the notational simplicity, we shall drop the index in the quantities if no confusion arises. For $S_{t-s} \phi(X(s))$,
applying It$\hat{o}$ formula to $X(s)$ and the equation \eqref{e:KEMFK} to
$S_{t-s}$, (differentiating on $s$), we have
\begin{equation*} \label{e:Ito}
\begin{split}
d\left[S_{t-s} \phi(X(s))\ek(s)\right]&=\mathcal{L}_m S_{t-s}
\phi(X(s))\ek(s)ds+D S_{t-s} \phi(X(s))\ek(s) QdW_s
\\
& \ \ \ -\mathcal{L}_m S_{t-s} \phi(X(s)) \ek(s)ds+K|AX(s)|^2
S_{t-s} \phi(X(s))
\ek(s)ds \\
& \ \ \ -S_{t-s}\phi(X(s))K|AX(s)|^2 \ek(s)ds
\\
&=D S_{t-s} \phi(X(s)) \ek(s) QdW_s
\end{split}
\end{equation*}
where $\mathcal{L}_m$ is the Kolmogorov operator defined in \eqref{e:KEM}, thus
\begin{equation} \label{e:Ito1}
\phi(X(t))\ek(t)=S_t\phi(x)+ \int_0^t D S_{t-s} \phi(X(s))\ek(s)
QdW_s
\end{equation}
Given any $h \in \pi_m H$, by (A3) of Assumption \ref{a:Q} and \eqref{e:Ito1}, we have
$y^\ph_t:=(Q^{\ph})^{-1}D_{h^\ph}X^\ph(t)$ so that
\begin{equation} \label{e:BD}
\begin{split}
& \ \ \ \E[\phi(X(t)) \ek(t)\int_{0}^{t/2} \Ll y^\ph_s,
dW^\ph_s \Rr]
\\
&=\E[\int_0^{t} DS_{t-s}\phi(X(s)) \ek(s)
Q^\ph dW_s^{\ph}
\int_{0}^{t/2} \Ll (Q^{\ph})^{-1}D_{h^\ph}X^\ph(s), dW^\ph_s \Rr] \\
&=\int_{0}^{t/2} \E[D^\ph S_{t-s} \phi(X(s))D_{h^\ph} X^\ph (s)
\ek(s)]ds,
\end{split}
\end{equation}
hence,
\begin{equation} \label{e:BD2}
\begin{split}
\int_{0}^{t/2} \E[D_{h^\ph} S_{t-s} \phi(X(s))\ek(s)]ds
&=\E[\phi(X(t)) \ek(t) \int_{0}^{t/2} \Ll (Q^{\ph})^{-1}D_{h^\ph}X^\ph(s),
dW^\ph_s \Rr]\\
& \ \ \ +\int_{0}^{t/2} \E[D^\pl S_{t-s} \phi(X(s))D_{h^\ph} X^\pl(s)
\ek(s)]ds.
\end{split}
\end{equation}
By the fact $S_t \phi(x)=\E[S_{t-s} \phi(X(s)) \ek(s)]$, \eqref{e:BD}
and \eqref{e:BD2}, we have
\begin{equation*} 
\begin{split}
D_{h^\ph}S_t \phi(x)&=\frac 2t \int_{0}^{t/2} D_{h^\ph}
\E[S_{t-s} \phi(X(s))\ek(s)]ds \\
&=
\frac 2t \E[\phi(X(t)) \ek(t) \int_{0}^{t/2} \Ll (Q^{\ph})^{-1}D_{h^\ph}X^\ph(s),
dW^\ph_s \Rr ] \\
& \ \ \ \ \ \ +\frac 2t \int_{0}^{t/2} \E[D^\pl S_{t-s} \phi(X(s))D_{h^\ph}
X^\pl(s) \ek(s)]ds \\
&\ \ \ \ \ \ \ -\frac {4K}{t} \int_{0}^{t/2} \E[S_{t-s} \phi(X(s))\ek(s)
\int_0^s \Ll AX(r),AD_{h^\ph} X(r) \Rr dr]ds \\
&\ =\frac 2tI_1+\frac 2t I_2-\frac{4K}t I_3.
\end{split}
\end{equation*}

We now fix $T, \gamma, k, r$ and let $C$ be constants depending on $T, \gamma, k$ and $r$
(whose values can vary from line to line), then $I_1,I_2$ and $I_3$ above
can be estimated as follows:
\begin{equation*}
\begin{split}
|I_1| & \leq ||\phi||_k \E \left[\mathcal{E}_{K/2}(t) (1+|AX(t)|^k)
\mathcal{E}_{K/2}(t) \int_{0}^{t/2} \Ll (Q^{\ph})^{-1}D_{h^\ph}X^\ph(s),
dW^\ph_s \Rr \right] \\
& \leq ||\phi||_{k} \E\left(\sup_{0 \leq s \leq T}(1+|AX(s)|^k)^2 \mcl
E_{K}(s)\right)^{\frac 12} \E \left(|\mathcal E_{\frac K2}(t)
\int_{0}^{t/2} \Ll (Q^{\ph})^{-1}D_{h^\ph}X^\ph(s),
dW^\ph_s \Rr|^2 \right)^{\frac 12} \\
& \leq C ||\phi||_{k} (1+|Ax|^k) \left[\E \left(\int_{0}^{t/2}
|A^rD_{h^\ph}X^\ph(s)|^2 \ek(s)ds \right) \right]^{1/2} \\
& \leq C t^{1/2-(r-\gamma)}||\phi||_k (1+|Ax|^k)|A^\gamma h|
\end{split}
\end{equation*}
where the last two inequalities are by \eqref{e:XmEmKEst}, \eqref{e:EkMat} and \eqref{e:ArDhX} in order.
By \eqref{e:DhHXL} and \eqref{e:XmEmKEst},
\begin{equation*}
\begin{split}
|I_2| & \leq \frac {C}{K} \int_{0}^{t/2}||A^{-\gamma}D^\pl S_{t-s} \phi||_{k}
\E \left [(1+|AX(s)|^k) \mathcal{E}_{K/2}(s) \right]ds |A^\gamma h| \\
& \leq \frac {C}{K}\int_{0}^{t/2} ||A^{-\gamma}DS_{t-s} \phi||_{k} ds (1+|Ax|^k)|A^\gamma h|.
\end{split}
\end{equation*}
By Markov property of $X(t)$ and \eqref{e:XmEmKEst}, we have
 \begin{equation*}
\begin{split}
|I_3| &=\int_{0}^{t/2} \E \left\{\E[\phi(X(t)) e^{-K \int_s^t |AX(r)|^2dr}|\mathcal{F}_s]
\ek(s)\int_0^s \Ll AX(r),AD_{h^\ph} X(r) \Rr dr \right \}ds \\
& \leq C ||\phi||_k \int_0^t \E[(1+|AX(s)|^k) \mcl E_{K/2}(s)
\int_0^s \mathcal{E}_{K/2} (r) |AX(r)| \cdot |AD_{h^\ph} X(r)|dr]ds,
\end{split}
\end{equation*}
moreover, and by H$\ddot{o}$lder inequality, Poincare inequality $|A^{\gamma+\frac 12} x| \geq |Ax|$, \eqref{e:IntEmk} and \eqref{e:DhX},
\begin{equation*}
\begin{split}
& \ \ \ \int_0^s \mathcal{E}_{K/2} (r) |AX(r)| \cdot |AD_{h^\ph} X(r)|dr \\
& \leq (\int_0^s \mathcal{E}_{K/2} (r) |AX(r)|^2 dr)^{\frac 12} (\int_0^s |AD_{h^\ph} X(r)|^2 \mathcal{E}_{K/2} (r) dr)^{\frac 12} \\
& \leq
[\int_0^s \mathcal{E}_{K/2}(r) |A^{\frac 12+\gamma} D_{h^\ph}
X(r)|^2dr]^{\frac12} \leq  |A^\gamma h|;
\end{split}
\end{equation*}
hence, by \eqref{e:XmEmKEst} and the above,
\begin{equation*}
|I_3| \leq Ct ||\phi||_k (1+|Ax|^k) |A^\gamma h|.
\end{equation*}
Collecting the estimates for $I_1$, $I_2$ and $I_3$, we have
\begin{equation} \label{e:HHL}
|D_{h^\ph}S_t \phi(x)| \leq
C \left\{(t^{-\frac 12-(r-\gamma)}+K)||\phi||_k+\frac{1}{Kt} \int_{0}^{t/2}
||A^{-\gamma}DS_{t-s} \phi||_{k} ds \right\} (1+|Ax|^k) |A^\gamma h|
\end{equation}
 For the low frequency part, according to Lemma
\ref{lem:LowDerivative}, we have
\begin{equation} \label{e:LLH}
\begin{split}
|D_{h^\pl}S_t \phi(x)|&=|D_{h^\pl}S_{t/2}(S_{t/2}\phi)(x)| \\
& \leq |\E[D^\ph S_{t/2}\phi(X(t/2)) D_{h^\pl} X^{\ph}(t/2)\ek(t/2)]| \\
& \ \ +|\E[D^\pl S_{t/2}\phi(X(t/2)) D_{h^\pl} X^\pl (t/2)\ek(t/2)]| \\
& \ \ +\E[|S_{t/2}\phi(X(t/2))|\ek(t/2) K \int_0^{t/2} |AX(s)| |AD_{h^\pl} X(s)|ds] \\
& \leq C \left\{\frac{1}{K} ||A^{-\gamma}DS_{t/2}\phi||_k+t^{-p}
e^{Ct} ||\phi||_{k}+K||\phi||_k\right\} (1+|Ax|^k)|A^\gamma h|
\end{split}
\end{equation}
where the last inequality is due to \eqref{e:DhLXH}, \eqref{e:XmEmKEst} and \eqref{e:LowDerivative},
and to the following estimate (which is obtained by the same argument as in estimating $I_3$):
\begin{equation*}
\begin{split}
\E[S_{t/2}\phi(X(t/2))\ek(t/2)  \int_0^{t/2} |AX(s)| |AD_{h^\pl} X(s)|ds]  \leq C||\phi||_k (1+|Ax|^k) |A^{\gamma} h|
\end{split}
\end{equation*}
\ \\

Denote $\alpha=p+\frac 12+r-\gamma$ and
$$\phi_{T}=\sup_{0 \leq t \leq T} t^{\alpha} ||A^{-\gamma}D S_t \phi||_k,$$
 by \eqref{e:DhX} and
the similar argument as estimating $I_3$, we have
\begin{equation*}
\begin{split}
|D_hS_t \phi(x)| &=|D_h \E[\phi(X(t)) \ek(t)]| \\
& \leq \E\left[|A^{-\gamma} D \phi(X(t))| \mcl E_{\frac K2} (t)|A^\gamma D_hX(t)| \mcl E_{\frac K2} (t) \right]
\\
& \ \ +2K\E \left[|\phi(X(t))| \mcl E_{\frac K2} (t)  \mcl E_{\frac K2} (t)\int_0^t |AX(s)||AD_hX(s)|ds \right] \\
& \leq C(T,K, \gamma, k) (||A^{-\gamma}
D\phi||_k+||\phi||_k)(1+|Ax|^k) |A^\gamma h|,
\end{split}
\end{equation*}
which implies $||A^{-\gamma} D S_t \phi||_k \leq C(T,K,
\gamma, k)(||A^{-\gamma}D \phi||_k+||\phi||_k)$, thus $\phi_T<\infty$.
\ \\

Combine (\ref{e:HHL}) and (\ref{e:LLH}), we have for every $t \in
[0,T]$
\begin{equation*}
\begin{split}
& \ \ t^\alpha||A^{-\gamma} D S_t \phi||_k \\
&\leq  C t^p ||\phi||_k+\frac {C}{K}t^{\alpha-1}
 \int_{0}^{t/2}(t-s)^{-\alpha}(t-s)^{\alpha}
||A^{-\gamma}DS_{t-s} \phi||_k ds \\
& \ \ +CKt^{\alpha}||\phi||_k+\frac{C}{K} t^{\alpha}||A^{-\gamma} DS_{t/2}\phi||_k+C t^{\alpha-p}
e^{Ct} ||\phi||_k  \\
& \leq C t^p ||\phi||_k+\phi_T \frac {C}{K}t^{\alpha-1}
 \int_{0}^{t/2}(t-s)^{-\alpha} ds  \\
& \ \ +CKt^{\alpha}||\phi||_k+\phi_T\frac{C}{K}+C
t^{\alpha-p}e^{Ct} ||\phi||_k \\
& \leq \phi_{T} \frac{C}{K}+KCe^{CT} ||\phi||_k,
\end{split}
\end{equation*}
this easily implies
$$\phi_T \leq \phi_T \frac{C}{K}+KCe^{CT} ||\phi||_k.$$
As $K>0$ is sufficiently large, we have for all $t \in [0,T]$
\begin{equation*}
t^{\alpha} ||A^{-\gamma}D S_t \phi||_k \leq \frac
K{1-C/K} C e^{CT}||\phi||_k,
\end{equation*}
from which we conclude the proof.
\end{proof}
\section{Malliavin Calculus} \label{s:MalCal}
\subsection{Some preliminary for Malliavin calculus}
Given a $v \in L^2_{loc}({\R}^{+}, \pi_m H)$, the Malliavin
derivative of $X_m(t)$ in direction $v$, denoted as $\mathcal{D}_v
X_m(t)$, is defined by
\begin{equation*}
\mathcal{D}_v X_m(t)=\lim \limits_{\epsilon \rightarrow 0}
\frac{X_m(t,W+\epsilon V)-X_m(t,W)}{\epsilon}
\end{equation*}
where $V(t)=\int_0^t v(s) ds$, provided the above limit exists. $v$ can be random and is adapted
with respect to the filtration generated by $W$.

Recall $\pi^{\pl}H=\pi_n H$ and
$Z_{\pl}(n)=[-n,n]^3 \setminus (0,0,0)$ with
$n_0<n<m$ to be determined in Proposition \ref{prop:modified Hormander}.
The Malliavin
derivatives on the low and high frequency parts of $X_m(t)$, denoted by
$\mathcal{D}_v X_m^{\pl}(t)$ and $\mathcal{D}_v X_m^{\ph}(t)$, can be
defined in a similar way as above. Moreover, $\mathcal{D}_v X_m^{\pl}(t)$ and
$\mathcal{D}_v X_m^{\ph}(t)$ satisfy the following two SPDEs
respectively:
\begin{equation} \label{e:low-Malliavin-derivative}
\begin{split}
\p_t \mathcal{D}_{v}X_m^{\pl}&+A \mathcal{D}_{v}X_m^{\pl}+\tilde
B_m^{\pl}(\mathcal{D}_{v}X_m^{\pl},X_m)+\tilde
B_m^{\pl}(\mathcal{D}_{v}X_m^{\ph},X_m)=Q_m^{\pl} v^\pl
\end{split}
\end{equation}
with $\mathcal{D}_{v}X_m^{\pl}(0)=0$, and
\begin{equation} \label{e:high-Malliavin-derivative}
\begin{split}
 \p_t \mathcal{D}_{v}X_m^{\ph} +A
\mathcal{D}_{v}X_m^{\ph}+\tilde B_m^{\ph}(\mathcal{D}_{v}X_m^{\pl},X_m)+
\tilde B_m^{\ph}(\mathcal{D}_{v}X_m^{\ph},X_m) =Q_m^{\ph} v^\ph
\end{split}
\end{equation}
with $\mathcal{D}_{v}X_m^{\ph}(0)=0$, where $\tilde B(u,v)=B(u,v)+B(v,u)$. Moreover, we define a flow between $s$ and $t$ by $J^m_{s,t} \ (s \leq t)$, where $J^m_{s,t} \in \mathcal{L}(\pi^\pl H,\pi^\pl H)$ satisfies the following equation: $\forall \ h \in \pi^\pl H$
\begin{equation} \label{e:Jt}
\p_t J^m_{s,t}h+A J^m_{s,t}h+\tilde B_m^{\pl}(J^m_{s,t} h,X_m(t))=0
\end{equation}
with $J^m_{s,s}=Id \in \mathcal{L}(\pi^\pl H,\pi^\pl H)$. It is easy to see that
the inverse $(J^m_{s,t})^{-1}$ exists and satisfies
\begin{equation} \label{e:J-1t}
\p_t (J^m_{s,t})^{-1}h-(J^m_{s,t})^{-1}[A h+\tilde B_m^{\pl}(h,
X_m(t))]=0.
\end{equation}
Simply writing $J^m_t=J^m_{0,t}$, clearly, $J^m_{s,t}=J^m_t (J^m_s)^{-1}.$
\ \\

We shall follow the ideas in section 6.1 of \cite{EH01} to develop a Malliavin calculus for $X_m$,
one of the key points for this approach is to find an adapted process $v \in L_{loc}({\R}^{+};\pi_m H)$  such that
\begin{equation} \label{e:v direction}
 Q_m^{\ph} v^\ph (t)=\tilde B_m^{\ph}(\mathcal{D}_{v}X_m^{\pl}(t),X_m(t)),
 \end{equation}
which, combining with \eqref{e:high-Malliavin-derivative}, implies $\mathcal{D}_v X_m^{\ph}(t)=0$ for all $t>0$. More precisely,
\begin{prop} \label{prop:v-direction} 
There exists some  $v \in L^2_{loc}({\R}^{+}; \pi_m H)$
satisfying (\ref{e:v direction}), and
\begin{equation*}
\mathcal{D}_{v} X_m^{\pl}(t)=J^m_t\int_0^t (J^m_{s})^{-1} Q_m^{\pl}
v^\pl(s) ds, \ \ \ \ \mathcal{D}_{v} X_m^{\ph}(t)=0.
\end{equation*}
\end{prop}
\begin{proof}
When $\mathcal{D}_{v}X_m^{\ph}(t)=0$ for all $t \geq 0$, the
equation (\ref{e:low-Malliavin-derivative}) is simplified to be
\begin{equation*} \label{e:vL}
\p_t \mathcal{D}_{v}X_m^{\pl}+[A
\mathcal{D}_{v}X_m^{\pl}+\tilde B_m^{\pl}(\mathcal{D}_{v}X_m^{\pl},X_m)]=Q_m^{\pl}
v^\pl
\end{equation*}
with $\mathcal{D}_{v}X_m^{\pl}(0)=0$, which is solved by
\begin{equation} \label{e:DvPhiL}
\mathcal{D}_{v} X_m^{\pl}(t)=\int_0^t J^m_{s,t} Q_m^{\pl} v^\pl(s)
ds=J^m_t\int_0^t (J^{m}_{s})^{-1} Q_m^{\pl} v^\pl(s) ds.
\end{equation}
Due to (A3) of Assumption \ref{a:Q}, there exists some $v \in L^2_{loc}({\R}^{+}, \pi_m H)$
so that $v^\ph$ satisfies (\ref{e:v direction}),
therefore, (\ref{e:high-Malliavin-derivative}) is a homogeneous
linear equation and has a unique solution
$\mathcal{D}_{v}X_m^{\ph}(t)=0, \ \ \forall \ t>0.$
\end{proof}
With the previous lemma, we see that the Malliavin derivative is essentially restricted in
\emph{low} frequency part. Take
$$N:=2[(2n+1)^3-1]$$
vectors $v_1,\ldots,v_{N}
\in L^2_{loc}({\R}^+;\pi_m H)$ with each satisfying Proposition
\ref{prop:v-direction} ($N$ is the dimension of $\pi^l H$). Denote
\begin{equation} \label{e:Malliavin-direction}
v=[v_1, \ldots, v_{N}],
\end{equation}
we have
\begin{equation} \label{e:DvMHL}
\mathcal{D}_{v} X_m^{\ph}=0, \ \
\mathcal{D}_{v} X_m^{\pl}(t)=J^m_t \int_0^t (J^{m}_s)^{-1} Q_m^{\pl} v^\pl (s) ds,
\end{equation}
where $Q_m^{\pl}$ is defined in \eqref{e:QlQhRep}. In particular,
$\mathcal{D}_{v} X_m^{\pl}(t)$ is an $N \times N$ matrix.
 Choose
\begin{equation*}
v^\pl (s)=[(J^{m}_s)^{-1} Q_m^{\pl}]^{*}
\end{equation*}
and denote
\begin{equation}
\mathcal{M}^m_t=\int_0^t [(J^{m}_s)^{-1} Q_m^{\pl}][(J^{m}_s)^{-1} Q_m^{\pl}]^{*} ds,
\end{equation}
$\mathcal{M}^m_t$ is called \emph{Malliavin matrix}, and is clearly
a symmetric operator in $\mathcal{L}(\pi^\pl H, \pi^\pl H)$.
$\forall \ \eta \in \pi^\pl H$, we have by  Parseval's identity
\begin{equation} \label{e:respresentation of Malliavin}
\begin{split}
\ \Ll \mathcal{M}_t \eta, \eta \Rr &=\int_0^t \Ll [(J^{m}_s)^{-1}
Q_m^{\pl}]^{*} \eta, [(J^{m}_s)^{-1}
Q_m^{\pl}]^{*} \eta \Rr ds \\
&=\sum \limits_{k \in Z_\pl(n)} \sum
\limits_{i=1}^2 \int_0^t |\Ll (J^{m}_s)^{-1} Q_m^{\pl} e^i_k, \eta \Rr|^2ds \\
&=\sum \limits_{k  \in Z_\pl(n) \setminus Z_\pl(n_0)}
\sum \limits_{i=1}^2 \int_0^t |\Ll (J^{m}_s)^{-1} q^i_ke_k, \eta
\Rr|^2ds
\end{split}
\end{equation}
where $q^i_k$ is the i-th column vector of the $2 \times 2$ matrix $q_k$ (recall \eqref{e:QRep}).  \\

The following lemma is crucial for proving Lemma \ref{lem:LowDerivative} and
will be proven in Appendix \ref{app:TecLem}.
\begin{lem} \label{l:JacEst}
1. For any $h \in \pi^\pl H$, we have
\begin{equation} \label{e:JacEst1}
|J^m_{t}h|^2 \ekm (t) \leq |h|^2,
\end{equation}
\begin{equation} \label{e:DhXpl}
|D_{h} X^\pl(t)|^2 \ekm (t) \leq |h|^2,
\end{equation}
\begin{equation} \label{e:J-1}
|(J^{m}_t)^{-1} h|^2 \ekm(t) \leq Ce^{Ct}|h|^2
\end{equation}
\begin{equation}   \label{e:JacEst2}
\left |\ekm(t) (J^{m}_{t})^{-1}-Id \right|_{\mathcal{L}(H)}\leq t^{1/2}Ce^{Ct}
\end{equation}
\begin{equation} \label{e:JmQlEmk}
\E \left(\int_0^t |[(J_s^{m})^{-1} Q_m^\pl]^{*} h|^2 \ekm(s)ds\right)
\leq t e^{Ct} tr[Q_m^{\pl} (Q_m^{\pl})^{*}] |h|.
\end{equation}
where the above $C=C(n)>0$ can vary from line to line and the $n$ is the size of $\pi^\pl H$ defined in \eqref{e:LowHigH}.  \\

2. Suppose that $v_1,v_2$ satisfy Proposition \ref{prop:v-direction}
and $h \in \pi^\pl H$, we have
\begin{equation} \label{e:Mal1OrdEst}
|A \mathcal{D}_{v_1}
X_m^{\pl}(t)|^2 \ekm(t) \leq C \int_0^t e^{1/2(t-s)}
|v^\pl_1(s)|^2 \mathcal{E}_{m,K}(s) ds
\end{equation}
\begin{equation} \label{e:Mix2OrdEst}
|\mathcal{D}_{v_1}D_h
X_m^{\pl}(t)|^2 \ekm(t) \leq  Ce^{Ct}
|h|^{2} \left(\int_0^t
|v^\pl_1(s)|^{2}\mathcal{E}_{m,\frac K2}(s) ds \right)
\end{equation}
\begin{equation} \label{e:Mal2OrdEst}
|\mathcal{D}^2_{v_1 v_2}
X_m^{\pl}(t)|^2 \ekm(t) \leq  Ce^{C t} \left(\int_0^t |v^\pl_1(s)|^2
\mathcal{E}_{m,\frac K2}(s)ds \right)
\left(\int_0^t
|v^\pl_2(s)|^{2} \mathcal{E}_{m,\frac K2}(s)ds \right)
\end{equation}
where the above $C=C(n)>0$ can vary from line to line and the $n$ is the size of $\pi^\pl H$ defined in \eqref{e:LowHigH}.
\end{lem}
\subsection{H\"{o}rmander's systems and proof of Lemma \ref{lem:LowDerivative}}   \label{sub:LowDerivative} 
We consider the SPDE
about $X_m^{\pl}$ in Stratanovich form as
\begin{equation} \label{e:LowFreqEqn}
dX_m^{\pl}+[A_m^{\pl} X_m^{\pl}+B_m^{\pl}(X)]dt=\sum \limits_{k \in Z_\pl (n)\setminus Z_\pl(n_0)} \sum
\limits_{i=1}^2 q^i_k \circ dw^i_k(t)e_k
\end{equation}
where $A^\pl$ is the Stokes operator restricted on $\pi^\pl H$ and $q_k^i$
is the $i$-th column vector in the $2 \times 2$ matrix $q_k$ (under the orthonormal basis ($e_{k,1}$,$e_{k,2}$) of $k^{\bot}$). Given any two Banach
spaces $B_1$ and $B_2$, denote $P(B_1,B_2)$ the collections of
functions from $B_1$ to $B_2$ with polynomial growth.  We
introduce the Lie bracket on $\pi^\pl H$ as follows: $\forall \
K_1 \in P(\pi_m H, \pi^\pl H), \ K_2 \in
P(\pi_m H, \pi^\pl H)$, define $[K_1,K_2]$ by
$$[K_1,K_2](x)=DK_1(x)K_2(x)-DK_2(x)K_1(x) \ \ \ \forall \ x \in \pi_m H.$$
The brackets $[K_1,K_2]$ will appear when \emph{differentiating $J_t^{-1}K_1(X(t))$} in the proof of Lemma
\ref{lem:inverse}.
\begin{defn} \label{defn:Hormander systems}
The H\"{o}mander's system {\bf K} for equation \eqref{e:LowFreqEqn} is defined
as follows: given any $y \in \pi_m H$, define
\begin{align*}
&{\bf K_0}(y)=\{q^i_k e_k; k \in Z_\pl(n) \setminus Z_\pl(n_0), i=1,2\}\\
&{\bf K_1}(y)=\{[A^\pl_my+B_m^{\pl}(y,y),q^i_{k} e_k]; k \in Z_\pl (n) \setminus Z_\pl(n_0), i=1,2\} \\
&{\bf K_2}(y)=\{[q^i_{k}e_k,K(y)]; K \in {\bf K_1}(y), k \in Z_\pl(n) \setminus Z_\pl(n_0),
i=1,2\}
\end{align*}
and ${\bf K}(y)=span\{{\bf K_0}(y) \cup {\bf K_1}(y) \cup {\bf K_2}(y)\}$, where each $q^i_k$ is
the column vector defined in \eqref{e:QRep}.
\end{defn}
\begin{defn}
The system ${\bf K}$
satisfies the \emph{restricted H\"{o}rmander condition} if
there exist some $\delta>0$ such that for all $y \in \pi_m H$
\begin{equation} \label{con:Hormander}
\sup \limits_{K \in {\bf K}}|\Ll K(y),\ell \Rr| \geq \delta
|\ell|, \ \ \ \ \ell \in \pi^\pl H.
\end{equation}
\end{defn}
\ \\

The following lemma gives some inscription for the elements in ${\bf K_2}$ (see \eqref{e:K2Ins}) and plays the key role for the proof of Proposition \ref{prop:modified Hormander}.
\begin{lem} \label{l:MixSet}
 For each $k \in Z_\pl(n_0)$, define \emph{mixing set} $Y_k$ by
$$Y_k=\left \{\tilde B_{m,k} \left (q_{j} \ell_j e_j, q_{l} \ell_l e_l\right):
j,l \in Z_\ph (n_0); \ell_j \in j^{\bot}, \ell_l \in l^{\bot}\right \},$$
where $\tilde B_{m,k}(x,y)$ is the Fourier coefficient of $\tilde B_m(x,y)$ at the mode $k$.
For all $k \in Z_\pl(n_0)$, $span \{Y_k\}=k^{\bot}$.
\end{lem}
\begin{prop} \label{prop:modified Hormander}
${\bf K}$ in Definition
\ref{defn:Hormander systems} satisfies the restricted H\"{o}rmander
condition.
\end{prop}
\begin{proof}
 It suffices to show that for each $k \in Z_\pl(n_0)$, the Lie brackets in Definition \ref{defn:Hormander systems} can produce at least two linearly independent vectors of $Y_k$ in Lemma \ref{l:MixSet}. (We note that \cite{R04} proved a similar proposition).

As $k \in Z_\pl(n_0) \cap {\Z}^3_{+}$, by Lemma \ref{l:MixSet}, $Y_k$  has at least
two linearly independent vectors $h^1_k \nparallel h^2_k$.
Without lose of generality,
assume $h^1_k=\tilde B_k(q^1_{j_k}e_{j_k},q^1_{l_k}e_{l_k})$
and $h^2_k=\tilde B_k(q^2_{j_k}e_{j_k},q^2_{l_k}e_{l_k})$ with
$j_k,-l_k \in Z_\ph(n_0)$ and $j_k+l_k=k$.
We can easily have (\emph{simply writing $j=j_k, l=l_k$})
\begin{equation} \label{e:K2Ins}
[q^i_je_j,[A^\pl y+B^\pl(y,
y),q^i_l e_l]]=-\tilde B^\pl(q^i_le_l,q^i_je_j),
\end{equation}
 and by \eqref{e:Bk1}-\eqref{e:BkZero},
$$[q^i_je_j,[A^\pl y+B^\pl(y,y),q^i_k e_k]]=-\frac{1}{2} \tilde B_{j-l}
(q^i_je_j,q^i_l e_l)-\frac12 \tilde B_k (q^i_je_j,q^i_l e_l).$$
Clearly, $j-l \in Z_\ph(n_0)$, by (A2) of Assumption \ref{a:Q},
$\tilde B_{j-l} (q^1_je_j,q^1_l e_l)$ and
$\tilde B_{j-l} (q^2_je_j,q^2_l e_l)$ must both be equal to
a linear combination of $q^i_{j-l}e_{j-l} \ (i=1,2)$.
Combining this observation with the fact $\tilde B_k(q^1_{j},q^1_{l}) \nparallel \tilde B_k(q^2_{j_k},q^2_{l_k})$, one immediately has that
$[q^i_je_j,[A^\pl y+B^\pl(y,y),q^i_l e_l]]$ $(i=1,2)$ and  $q^i_{j-l}e_{j-l}$ $(i=1,2)$
span $k^{\bot}$.

Similarly, we have the same conclusion for $k \in Z_\pl(n_0) \cap {\Z}^3_{-}$. Choose the $n$ in \eqref{e:LowHigH} sufficiently large so that $j_k, l_k, j_k+l_k,
j_k-l_k \in Z_\pl(n)$ for all $k \in Z_\pl(n_0)$.
\end{proof}
With Proposition \ref{prop:modified Hormander}, we can show the following key lemma (see the proof in Section \ref{ss:Proof of Lemma}).
\begin{lem} \label{lem:inverse}
Suppose that $X_m(t,x)$ is the solution to
equation \eqref{e:GalerkinM} with initial data $x \in \pi_m H$.
Then $\mathcal{M}^m_t$ is invertible almost surely. Denote $\lambda_{min}(t)$ the minimal eigenvalue of $\mathcal{M}^m_t$, then there exists some constant $q>0$ (possibly very large), for all $p>0$, we have a constant $C=C(p)>0$ such that
\begin{equation} \label{e:MinEigVal}
\mathbb P \left \{ \frac{1}{\lambda_{min}(t)} \geq \frac 1 {\e^q} \right\} \leq  \frac{C \e^p}{t^p}.
\end{equation}
\end{lem}
\ \ \

\begin{proof} [Proof of Lemma \ref{lem:LowDerivative}] We shall simply write $X(t)=X_m(t),
J_t=J^m_t$, $\mathcal{M}_t=\mathcal{M}^m_t$,
$Q^\pl=Q_m^{\pl}$ and $\ek(t)=\ekm(t)$ for the notational simplicity. Under an orthonormal basis of $\pi^\pl H$, the operators
$J_t$, $\mathcal{M}_t$, $\mcl D_v X^\pl(t)$ with $v$ defined in \eqref{e:Malliavin-direction},
and $D^\pl X^\pl(t)$ can all be
represented by $N \times N$ matrices, where $N$ is the dimension of $\pi^\pl H$. Noticing $\mcl D_v X^\pl (t)=J_t \mcl M_t$ (see \eqref{e:DvMHL}), the following $\phi_{il}$ is well defined:
$$\phi_{il}(X(t))=\phi(X(t)) \sum_{j=1}^N [(\mcl D_v X^\pl (t))^{-1}]_{ij} [D^\pl X^\pl (t)]_{jl} \ek(t) \ \ \ i, l=1, \ldots,N,$$
where $v$ is defined in \eqref{e:Malliavin-direction} with $v^\pl(t)=(J^{-1}_tQ^\pl)^{*}$.
For any $h \in \pi^\pl H$, by our special choice of $v$, we have
\begin{equation} \label{e:Malliavin-inverse}
\begin{split}
\mathcal{D}_{vh}\phi_{il}(X(t))&=D^\pl \phi(X(t))
[\mathcal{D}_{v} X^\pl(t)h]\sum_{j=1}^N [(\mcl D_v X^\pl (t))^{-1}]_{ij} [D^\pl X^\pl (t)]_{jl} \ek(t) \\
& \ +\phi(X(t)) \sum_{j=1}^N  \mathcal{D}_{vh}\left\{[(\mcl D_v X^\pl (t))^{-1}]_{ij} [D^\pl X^\pl (t)]_{jl}\right\} \mathcal{E}_K(t) \\
& \ -2K\phi_{il}(X(t)) \int_0^t \Ll AX(s),A\mathcal{D}_{vh}X(s) \Rr ds
\end{split}
\end{equation}
Note that $\pi^\pl H$ is isomorphic to ${\R}^{N}$ under the orthonormal basis.
Take the standard orthonormal basis $\{h_i;i=1,\ldots,N\}$ of ${\R}^{N}$,
which is a representation of the orthonormal
basis of $\pi^\pl H$. Set $h=h_i$ in (\ref{e:Malliavin-inverse}) and sum over
$i$, we obtain
\begin{equation} \label{e:Malliavin-inverse 1}
\begin{split}
& \ \ \ \E \left(D^\pl \phi(X(t)) D_{h_l}^\pl X^\pl (t) \mathcal{E}_K(t)\right) \\
&=\E \left(\sum \limits_{i=1}^{N}
\mathcal{D}_{vh_i} \phi_{il}(X(t))
\right)-\E \left(\sum \limits_{i,j=1}^{N} \phi(X(t))
\mathcal{D}_{vh_i}\left\{[(\mcl D_v X^\pl (t))^{-1}]_{ij} [D^\pl X^\pl (t)]_{jl}\right\}  \mathcal{E}_K(t) \right) \\
& \ \ \ +2K \E \left(\sum \limits_{i=1}^{N} \phi_{il}(X(t))
\int_0^t \Ll AX(s),A\mathcal{D}_{vh_i}X(s) \Rr ds \right)
\end{split}
\end{equation}
Let us first bound the first term on the r.h.s. of (\ref{e:Malliavin-inverse 1})
as follows: By Bismut formula (simply write $v_i=vh_i$), \eqref{e:XmEmKEst}
and the identity $\mathcal{D}_{v} X^\pl(t)=J_t \mathcal{M}_t$, one has
\begin{equation} \label{e:ExpBis}
\begin{split}
& \ \ |\E \left(\sum \limits_{i=1}^{N} \mathcal{D}_{v_i} \phi_{il}(X(t))
\right)|=|\E \left(\sum
\limits_{i,j=1}^{N} \phi(X(t))[\mcl M_t^{-1} J^{-1}_t]_{ij} [D^\pl X^\pl (t)]_{jl} \mathcal{E}_K(t) \int_0^t
\Ll v^\pl_i,dW_s \Rr \right)| \\
&\leq C ||\phi||_k(1+|Ax|^k) \sum_{i,j=1}^N \E\left(\frac {\mathcal{E}_{K/2}(t)} {\lambda_{min}} |J^{-1}_t h_j| |D^\pl_{h_l} X^\pl (t)| |\int_0^t
\Ll v^\pl_i,dW_s \Rr|\right),
\end{split}
\end{equation}
moreover, by H$\ddot{o}$lder's inequality, Burkholder-Davis-Gundy's inequality, \eqref{e:MinEigVal}, \eqref{e:J-1}, \eqref{e:DhXpl} and \eqref{e:JmQlEmk} in order,
\begin{equation} \label{e:EkJ-1Dhl}
\begin{split}
& \ \E\left(\frac {\mathcal{E}_{K/2}(t)} {\lambda_{min}} |J^{-1}_t h_j| |D^\pl_{h_l} X^\pl (t)| |\int_0^t
\Ll v^\pl_i,dW_s \Rr|\right) \\
& \leq \left [\E
\left(\frac{1}{\lambda_{min}^6}\right)\right]^{\frac16}
\left[\E \left(|J^{-1}_t h_j|^6 \ek(t)\right)\right]^{\frac16}
\left[\E \left(|D^\pl_{h_l} X^\pl (t)|^6 \ek(t)\right)\right]^{\frac16}
\left[\E(\int_0^t \mathcal{E}_{\frac K3}(s) |(J^{-1}_sQ^\pl )^{*}h_i|^2 ds)\right]^{\frac 12} \\
& \leq \frac{Ce^{Ct}}{t^{p}}  \ \ \
\end{split}
\end{equation}
where $p>6q+1$ and $C=C(p)$.
Combining \eqref{e:ExpBis} and \eqref{e:EkJ-1Dhl}, we have
\begin{equation}
|\E \left(\sum \limits_{i=1}^{N} \mathcal{D}_{v_i} \phi_{il}(X(t))
\right)|
\leq \frac{Ce^{Ct}}{t^p} ||\phi||_k(1+|Ax|^k)
\end{equation}
where $C=C(p, k)>0$. By the similar method but more complicate calculations (using Lemma \ref{lem:inverse} and the estimates
in Lemma \ref{l:JacEst}), we have the same bounds
for the other two terms on the r.h.s. of \eqref{e:Malliavin-inverse 1}. Hence,
$$|\E \left[D^\pl \phi(X(t)) D_{h_l}^\pl X^\pl(t) \ek(t)\right]| \leq t^{-p} Ce^{C t}||\phi||_k(1+|Ax|^k)$$
for all $t>0$. Since the above argument is in the frame of $\pi^\pl D(A)$ with the orthonormal base
$\{h_l; 1 \leq l \leq N\}$, we have
$$|\E \left[D^\pl \phi(X(t)) D_h X^\pl(t) \ek(t)\right]| \leq t^{-p} C e^{C t} ||\phi||_k(1+|Ax|^k) |h|
\ \ \ \ h \in \pi^\pl H.$$
\end{proof}
\section{Proof of the main theorems} \label{s:ProMaiThm}
\subsection{Gradient estimates of $u_m(t)$}
To prove the strong Feller of the semigroup $P^m_t$ (recall $P^m_t
\phi=u_m(t)$) and the later limiting semigroup $P_t$, a typical method
is to show that $P^m_t$ has a gradient estimate similar to
\eqref{e:StEst}. In \cite{DO06}, one has the same 
estimate as \eqref{e:StEst} \emph{but} with $\alpha=\frac 12+r-\gamma$ therein, thanks to
the property $0<\frac 12+r-\gamma<1$, one can easily show
$$||A^{-\gamma}Du_m(t)||_2 \leq C(t^{-\frac 12-r+\gamma}+1)
||\phi||_0,$$ this is exactly the second
inequality in Proposition 3.5 of \cite{DO06}. \\

In our case, by the same method as in \cite{DO06} (i.e. applying
\eqref{e:StEst} to bound the r.h.s. of \eqref{e:DuhUmt}), we \emph{formally} have
\begin{equation} \label{e:WroDuh}
||A^{-\gamma} D u_m(t)||_2 \leq C t^{-\alpha} ||\phi||_0+KC\int_0^t
(t-s)^{-\alpha} ds ||\phi||_0,
\end{equation}
however, the integral on the r.h.s. of \eqref{e:WroDuh} blows up
due to $\alpha>1$ in \eqref{e:StEst} . \\

We have two ways to overcome the problem of not integrability in
\eqref{e:WroDuh}. One is by an interpolation argument (see Proposition \ref{p:LesGam}),
the other is by some more delicate analysis (see Proposition \ref{p:HigAx}). The
underlying ideas of the two methods are the same, i.e. trading
off the regularity of the space for the integrability of the time. \\

\begin{prop} \label{p:LesGam}
Given $T>0$, for any $0<t \leq T$, $\max\{\frac 12,r-\frac 12\}<\gamma \leq 1$ and $0<\beta<1$, if $\phi \in C^1_b(D(A),\R)$, then
$S^m_t \phi$ and $u^m_t$ are both functions in
$C^{\beta/\alpha}_{2,\gamma}(D(A),\R)$, which is the H$\ddot o$lder space defined by \eqref{e:HolSpa}. Moreover,
\begin{equation} \label{e:LesGamSmt}
||S^m_t \phi||_{C^{\beta/\alpha}_{2,\gamma}} \leq C t^{-\beta} ||\phi||_{2},
\end{equation}
\begin{equation} \label{e:LesGam}
||u_m(t)||_{C^{\beta/\alpha}_{2,\gamma}} \leq C (t^{-\beta}+1) ||\phi||_{0},
\end{equation}
where
$\alpha=p+\frac 12+r-\gamma$ is defined in \eqref{e:StEst} and $C=C(T,\alpha,\beta,\gamma)>0$.
\end{prop}
\begin{proof}
On the one hand, for any $x \in D(A)$, by \eqref{e:StmDef} and \eqref{e:XmEmKEst}, one clearly has
\begin{equation*}
\begin{split}
|S^m_t \phi(x)|\leq ||\phi||_2 \E[(1+|AX^m(t)|^2) \ekm(t)] \leq C(1+|Ax|^2) ||\phi||_2
\end{split}
\end{equation*}
where $C>0$ is independent of $m$, $t$ and $x$. Hence, $S^m_t: C_2(D(A),\R) \rightarrow C_{2}(D(A),\R)$
has the estimate
$$||S^m_t \phi||_2 \leq C||\phi||_2.$$
On the other hand, by \eqref{e:StEst}, one has
$S^m_t: C_2(D(A),\R) \rightarrow C^1_{2,\gamma} (D(A),\R)$
with $$||S^m_t \phi||_{C^1_{2,\gamma}} \leq C t^{-\alpha} ||\phi||_2.$$
By a simple calculation with the the above two estimates, we have 
$$||S^m_t \phi||_{C^{\beta/\alpha}_{2,\gamma}} \leq C||S^m_t \phi||^{\beta/\alpha}_{C^{1}_{2,\gamma}}||S^m_t \phi||^{1-\beta/\alpha}_2 \leq C t^{-\beta} ||\phi||_2,$$
for any $0 \leq \beta \leq \alpha$. Take any $0<\beta<1$, applying the above estimate
on the Duhamel formula \eqref{e:DuhUmt} and the clear fact $||u_m(t)||_0 \leq ||\phi||_0$, we immediately have \eqref{e:LesGam}. 
\end{proof}
\begin{prop} \label{p:HigAx}
Given any $T>0$, there exists some
$C=C(T,\alpha,\gamma)>0$ such that
\begin{equation} \label{e:HigAx}
||A^{-\gamma} Du_m(t)||_{2\alpha} \leq C t^{-\alpha} ||\phi||_0
\end{equation}
where $\max\{r-\frac 12, \frac 12\}<\gamma \leq 1$ with $\gamma \neq \frac 34$.
\end{prop}
\begin{proof}
The idea of the proof is to split the integral $\int_0^t |D_h S^m_{t-s}(|Ax|^2 u_m(s))|ds$
into two pieces, '$\int_0^{\beta t} \cdots$' and '$\int_{\beta t}^t \cdots$' with some special
$\beta \in (0,1)$, applying \eqref{e:StEst} to the first piece and the probability presentation
of $S^m_{t-s}$ to the other. Roughly speaking, '$\int_0^{\beta t} \cdots$'
takes away the singularity of $(t-s)^{-\alpha}$ at $s=t$, while '$\int_{\beta t}^t \cdots$'
conquers the extra polynomial growth of $|Ax|^2$ in $S^m_{t-s}[|Ax|^2 u_m(s)]$. However,
we have to pay a price of an extra polynomial growth of $|Ax|^{2\alpha}$ for $Du^m(t)$. \\

For the notational simplicity, we shall
drop the index $m$ of the quantities if no confusions arise. Denote
$$\beta=1-\frac{1}{K^2(1+|Ax|^2)},$$
by \eqref{e:StEst} with $k=2$, we have
\begin{equation*}
\begin{split}
|A^{-\gamma} Du(t,x)| & \leq Ct^{-\alpha} ||\phi||_2 (1+|Ax|^2)+KC
\int_0^{\beta t} (t-s)^{-\alpha} ds ||\phi||_0(1+|Ax|^2)\\
&\ \ +K\int_{\beta t}^t |A^{-\gamma} D S_{t-s}(|Ax|^2
u(s))|ds \\
& \leq Ct^{-\alpha} ||\phi||_0 (1+|Ax|^2)+K^{2 \alpha+1}C t^{-\alpha+1}
(1+|Ax|^{2 \alpha})||\phi||_0 \\
&\ \ +K\int_{\beta t}^t |A^{-\gamma} D S_{t-s}(|Ax|^2
u(s))|ds,
\end{split}
\end{equation*}
thus
\begin{equation} \label{e:TAlpADhu}
\begin{split}
t^\alpha |A^{-\gamma} D_h u(t,x)| & \leq C||\phi||_0 (1+|Ax|^2)+K^{2 \alpha+1}Ct ||\phi||_0 (1+|Ax|^{2 \alpha}) \\
& \ \ \ +Kt^\alpha \int_{\beta t}^t |A^{-\gamma} D S_{t-s}(|Ax|^2
u(s))|ds.
\end{split}
\end{equation}
Define
$$u_{\phi,T}=\sup_{0 \leq s \leq T} s^\alpha ||A^{-\gamma} Du(s)||_{2 \alpha},$$
let us estimate the integral on the r.h.s. of \eqref{e:TAlpADhu} in the following way: it is easy to see that
\begin{equation} \label{e:IntBilEst}
\begin{split}
& \ \ \int_{\beta t}^t |D_h S_{t-s}(|Ax|^2
u(s)]|ds \\
&=\int_{\beta t}^t |\E(D_h |AX(t-s)|^2
u(s,X(t-s)) \ek(t-s)]|ds \\
& \ \ +\int_{\beta t}^t |\E(|AX(t-s)|^2
u(s,X(t-s)) D_h\ek(t-s)]|ds\\
& \ \ +\int_{\beta t}^t |\E(|AX(t-s)|^2
D_h u(s,X(t-s)) \ek(t-s)]|ds.
\end{split}
\end{equation}
By the same argument as estimating $I_3$ in the proof of Theorem
\ref{thm:StBound} and the easy fact $||u(t)||_0 \leq ||\phi||_0$ for
all $t \geq 0$, the first two integrals on
the r.h.s. of \eqref{e:IntBilEst} can both be bounded by
$$C(1+|Ax|^2) ||\phi||_0 |A^\gamma h|.$$
The last integral can be estimated as follows:
By \eqref{e:XmEmKEst}, \eqref{e:DhX} and the definition of $u_{T,\phi}$, one has
\begin{equation*}
\begin{split}
& \ \ \int_{\beta t}^t |\E(|AX(t-s)|^2
D_h u(s,X(t-s)) \ek(t-s)]|ds \\
& \leq \int_{\beta t}^t \E[(1+|AX(t-s)|^{2+2\alpha}) \mcl E_{\frac K2}(t-s)
||A^{-\gamma} D u(s)||_{2\alpha} \mcl E_{\frac K2}(t-s) |A^\gamma D_h X(t-s)|]ds \\
& \leq  C (1+|Ax|^{2+2\alpha}) |A^\gamma h| \int_{\beta t}^t ||A^{-\gamma} D u(s)||_{2\alpha} ds \\
& \leq C (1+|Ax|^{2+2\alpha}) |A^\gamma h| \left(\int_{\beta t}^t
s^{-\alpha} ds \right) u_{T, \phi} \\
& \leq \frac{C t^{-\alpha+1}}{K^2}
u_{T,\phi} (1+|Ax|^{2 \alpha}) |A^\gamma h|.
\end{split}
\end{equation*}
Collecting the above three estimates, we have
\begin{equation*}
\begin{split}
\int_{\beta t}^t |A^{-\gamma} DS_{t-s}(|Ax|^2
u(s))|ds \leq C(1+|Ax|^2) ||\phi||_0+\frac{Ct^{-\alpha+1}}{K^2} u_{T,\phi} (1+|Ax|^{2 \alpha}).
\end{split}
\end{equation*}
Plugging this estimate into \eqref{e:TAlpADhu} and dividing the both sides of the inequality by
$(1+|Ax|^{2+2 \alpha})$, one has
$$u_{T,\phi} \leq C||\phi||_0+C K^{2 \alpha+1}T ||\phi||_0+CKT^\alpha ||\phi||_0+\frac{CT}{K} u_{T,\phi}.$$
As $K>0$ is sufficiently large, 
$$u_{T, \phi} \leq \frac{C(1+K^{2\alpha+1}T+KT^\alpha)}{1-CT/K}||\phi||_0,$$
from this inequality, we immediately have \eqref{e:HigAx}.
\end{proof}
\subsection{Proof of Theorem \ref{t:MaiThm1}}
One can pass to the Galerkin approximation limit of $u_m(t)$ by the
same procedures as in \cite{DO06}. For the completeness, we
sketch out the main steps as following. \\

The following proposition is
nearly the same as Proposition 3.6 in \cite{DO06}, only with a small
modification in which \eqref{e:LesGam} plays an essential role.
\begin{prop} \label{e:EquConGalApp}
Let $\phi \in C^1_b(D(A), \R)$ and $T>0$. For any $0<\beta<1/2$,  $t_1 \geq t_2>0$, $m \in \mathbb N$ and $x \in D(A)$, we have some $C(T,\beta)>0$ such that
\begin{equation} \label{e:EquCon}
|u_m(t_1,x)-u_m(t_2,x)| \leq C||\phi||_{C^1_{2,1}}(1+|Ax|^6)(|t_2-t_1|^{\beta}+|A(e^{-At_2}-e^{-At_1})x|).
\end{equation}
\end{prop}
\ \

Define
$K_{R}=\{x \in D(A); |Ax| \leq R\}$, which is compact in $D(A^{\gamma})$ for any $\gamma<1$, we have
the following lemma (which is Lemma 4.1
in \cite{DO06}) by applying  Proposition \ref{e:EquConGalApp}.
\begin{lem} \label{l:GalLim}
Assume $\phi \in C^1_b(D(A), \R)$, then there exists a subsequence $(u_{m_k})_{k \in
\mathbb N}$ of $(u_m)$ and a function $u$ on $[0,T] \times D(A)$, such that
\begin{enumerate}
\item $u \in C_b((0,T]\times D(A))$ and for all $\delta>0$ and $R>0$
$$\lim_{k \rightarrow \infty} u_{m_k}(t,x)=u(t,x) \ \ \rm{uniformly
\ on \ [\delta, T] \times K_R.}$$
\item For any $x \in D(A)$, $u(\cdot,x)$ is continuous on $[0,T]$.
\item For any $\max\{\frac 12,r-\frac 12\}<\gamma \leq 1$ with $\gamma \neq \frac 34$, $\delta>0$, $R>0$ and $\beta<min\{1/2,\sigma/2\}$,
 there exists some $C=C(\gamma, \beta, \delta, R,T, \phi)$ such that for
 any $x, y\in K_R$, $t \geq s \geq \delta$,
 \begin{equation*}
 |u(t,x)-u(s,y)| \leq C(|A^{\gamma} (x-y)|+|t-s|^{\beta}).
 \end{equation*}
 \item For any $t \in [0,T]$, $u(t,\cdot) \in C_b(D(A),\R)$.
 \item $u(0)=\phi$.
\end{enumerate}
\end{lem}
\begin{lem} \label{l:TigLem}
For any $\delta \in (1/2,1+\sigma]$, there exists some constant $C(\delta)>0$ such that for any
$x \in H$, $m \in \N$, and $t \in [0,T]$, we have
\begin{enumerate}
\item $\E[|X_m(t,x)|^2]+\E\int_0^t |A^{1/2} X_m(s,x)|^2 ds \leq |x|^2+tr(QQ^{*})t.$
\item $\E \int_0^T \frac{|A^{\frac{1+\delta}2} X_m(s,x)|^2}{(1+|A^{\frac{\delta}2}
X_m(s)|^2)^{\gamma_\delta}}ds \leq C(\delta)$, with
$\gamma_\delta=\frac 2{2\delta-1}$ if $\delta \leq 1$ and
$\gamma_\delta=\frac {2\delta+1}{2\delta-1}$ if $\delta>1$.
\end{enumerate}
\end{lem}
By (1) of Lemma \ref{l:TigLem}, we can prove that the laws $\mathcal{L}(X_m(\cdot,x))$
is tight in $L^2([0,T],D(A^{s/2}))$
for $s<1$ and in $C([0,T],D(A^{-\alpha}))$ for $\alpha>0$. By Skohorod's embedding
Theorem, one can construct
a probability space $(\Omega_x, \mathcal F_x, \PP_x)$ with a
random variable $X(\cdot,x)$ valued in $L^2([0,T],D(A^{s/2})) \cap C([0,T],D(A^{-\alpha}))$
such that for any $x \in D(A)$ there exists some subsequence $\{X_{m_k}\}$ satisfying
\begin{equation} \label{e:WeaLimXm}
X_{m_k}(\cdot,x) \rightarrow X(\cdot,x) \ \ \ \ \ d\PP_x \ a.s.
\end{equation}
in $L^2([0,T],D(A^{s/2})) \cap C([0,T],D(A^{-\alpha}))$.
Moreover, by (3) of Lemma \ref{l:TigLem}, for $x \in D(A)$ we have (see (7.7) in \cite{DPD03})
\begin{equation}  \label{e:WeaLimXmA}
X_{m_k}(t,x) \rightarrow X(t,x) \ \ \ {\rm in } \ D(A) \ \ dt \times d\PP_x  \ a.s. \ [0,T] \times \Omega_x.
\end{equation}
\ \

Note that the subsequence $\{u_{m_k}\}$ in Lemma \ref{l:GalLim} depends on
$\phi$, by the separable property of $C(D(A),\R)$, we can find a subsequence $\{m_k\}$ of $\{m\}$,
independent of $\phi$, such that
$\{u_{m_k}\}_k$ converges. That is, we have the following lemma, which is Lemma 7.5 of \cite{DPD03}.
\begin{lem} \label{l:ComSubSeq}
There exists a subsequence $\{m_k\}$ of $\{m\}$ so that for any $\phi \in C^1_b(D(A), \R)$,
one has a function $u^{\phi} \in C_b([0,T] \times D(A))$ satisfying
\begin{equation} \label{e:ComLimUmk}
\lim_{k \rightarrow \infty} u^{\phi}_{m_k}(t,x)=u^{\phi}(t,x) \ \ \ for \ all \
(t,x) \in (0,T]\times D(A)
\end{equation}
and
\begin{equation*}
u^{\phi}_{m_k}(t,x) \rightarrow u^{\phi}(t,x) \ \ \ {\rm uniformly \ in \ [\delta,T] \times K_R \ for \
any \ \delta>0, R>0}.
\end{equation*}
where $u^{\phi}_{m_k}(t,x)=\E[\phi(X_{m_k}(t,x))]$.
\end{lem}
Take the subsequence $\{m_k\}$ in Lemma
\ref{l:ComSubSeq} and define
\begin{equation} \label{e:SemGro1}
P_t \phi(x)=u^{\phi} (t,x).
\end{equation}
for all $(t,x) \in [0,T] \times D(A)$, where $u^{\phi}$ is defined by
\eqref{e:ComLimUmk}. By Riesz Representation Theorem for functionals (\cite{Fo99}, page 223) and the easy fact $P_t {\bf 1}=1$, \eqref{e:SemGro1}
determines a unique probability
measure $P^{*}_t \delta_x$ supported on $D(A)$.  By \eqref{e:WeaLimXm},
for any $x \in D(A)$, we have some subsequence $\{m^x_k\}$
of $\{m_k\}$ so that $X_{m^x_k}(\cdot,x) \rightarrow X(\cdot,x)$ in
$L^2([0,T],D(A^{s/2})) \cap C([0,T],D(A^{-\alpha}))$
a.s. $d \PP_x$, hence
$$P_t\phi(x)=\E_x[\phi(X(t,x))]$$
for all $\phi \in C^1_b(D(A), \R) \cap C_b(D(A^{-\alpha}),\R)$.
Since the measure $P^{*}_t \delta_x$ is supported on $D(A)$, $\PP_x(X(t,x)  \in D(A))=1$, which
is (1) of Definition \ref{d:MarFam}.
By a classic approximation ($\mcl B_b(D(A),\R)$ can be approximated by $C(D(A),\R)$), we have
\begin{equation} \label{e:PtOnBouFun}
P_t \phi(x)=\E_x[\phi(X(t,x))] \ \rm{is \ well \ defined \ for \ all} \ \phi \in \mathcal B_b(D(A),\R).
\end{equation}
With the above observation, we can easily prove Theorem \ref{t:MaiThm1} as follows:
\begin{proof} [{Proof of Theorem \ref{t:MaiThm1}}]
Since $X_{m_k}(\cdot,x) \rightarrow X(\cdot,x)$ a.s. $\PP_x$
in $C([0,T],D(A^{-\alpha}))$ and the map $x \rightarrow \mathcal P^{m_k}_x$
is measurable ($\mathcal P^{m_k}_x$ is the law of $X_{m_k}(\cdot,x)$),
the map $x \rightarrow \mathcal P_x$ is also measurable. The following lemma is exactly
Lemma 4.5 in \cite{DO06} and expressed as
\begin{lem}
Let $X(\cdot,x)$ be the limit process of a subsequence
$\{X_{m_k}\}_k$. Then, for any $M,N \in \N$, $t_1, \cdots, t_n \geq 0$ and $(f_k)_{k=0}^M$
with each $f_k \in C^{\infty}_c(\pi_N H,\R)$, we have
\begin{equation} \label{e:MarOnGooFun}
\E_x[f_0(X(0,x)) f_1(X(t_1,x)) \cdots f_M(X(t_1+\cdots +t_M,x))]=
f_0(x)P_{t_1}[f_1 P_{t_2} (f_2P_{t_3}f_3 \cdots )](x)
\end{equation}
where each $f_k(x)=f_k(\pi_Nx)$ and $P_t$ is defined by \eqref{e:PtOnBouFun}.
\end{lem}
One can easy extend \eqref{e:MarOnGooFun} from $C^{\infty}_c(\pi_N H,\R)$ to
$\mathcal B_b(D(A),\R)$, which easily implies the Markov property of
the family $(\Omega_x,\mathcal F_x, \PP_x, X(\cdot,x))_{x \in D(A)}$.
\end{proof}
\subsection{Proof of Theorem \ref{t:MaiThm2}} \label{s:Erg}
To prove the ergodicity, we first prove that \eqref{e:NSEabs} has at least one invariant measure, and then show the uniqueness by Doob's Theorem. With the ergodic measure, we follow the
coupling method in \cite{Od07} to prove the exponential mixing property \eqref{e:ExpMix}.
\begin{lem} \label{l:ErgDynM}
Each approximate stochastic dynamics $X_{m}(t)$ has a unique
invariant measure $\nu_{m}$.
\end{lem}
\begin{proof}
By Proposition \ref{p:LesGam} or Proposition \ref{p:HigAx}, we can
easily obtain that $P^m_t$ is strong Feller. The existence of the
invariant measures for $X^m(t)$ is standard (see \cite{DPZ96}), and
it is easy to prove that $0$ is the support of each invariant
measure (see Lemma 3.1, \cite{EM01}). Therefore, by Corollary 3.17
of \cite{HM06}, we conclude the proof.
\end{proof}

The following lemma is the same as Lemma 7.6 in \cite{DPD03} (or Lemma 5.1 in \cite{DO06}).
\begin{lem}
There exists some constant $C>0$ so that
\begin{equation}
\int_H [|Ax|^2+|A^{1/2}x|^{2/3}+|A^{1+\sigma/2}x|^{(1+\sigma)/(10+8\sigma)}] \nu_{m_k}(dx)<C
\end{equation}
where $\sigma>0$ is the same as in Assumption \ref{a:Q}.
\end{lem}
With the above lemma, it is easy to see that $\{\nu_{m_k}\}$ is tight on $D(A)$, and therefore
there exists a limit measure $\nu$ which satisfies $\nu(D(A))=1$.
Taking any $\phi \in C_b^1(D(A), \R)$, we can check via the Galerkin approximation (or see the detail in pp. 938 of \cite{DPD03})
that
\begin{equation}
\int_H P_t \phi(x) \nu(dx)=\int_H \phi(x) \nu(dx)
\end{equation}
for any $t>0$. Hence $\nu$ is an invariant measure of $P_t$.

\begin{prop}
The system $X(t)$ is irreducible on $D(A)$. More precisely, for any $x,y \in D(A)$, we have
\begin{equation}
P_t [1_{B_\delta(y)}](x)>0.
\end{equation}
for arbitrary $\delta>0$, where $B_\delta(y)=\{z \in D(A); |Az-Ay| \leq \delta\}$.
\end{prop}
\begin{proof}
We first prove that the following control problem is solvable:
 Given any $T>0$, $x,y \in D(A)$ and $\e>0$, there exist $\rho_0=\rho_0(|Ax|,|Ay|,T)$,
$u$ and $w$ such that
\begin{itemize}
\item $w\in L^2([0,T];H)$ and $u\in C([0,T];D(A))$,
\item $u(0)=x$ and $|Au(T)-Ay|\leq\e$,
\item $\sup_{t\in[0,T]} |Au(t)|\leq\rho_0$,
\end{itemize}
and $u$, $w$ solve the following problem,
\begin{equation}\label{e:control}
\partial_t u + Au + B(u,u) = Qw,
\end{equation}
where $Q$ is defined in Assumption \ref{a:Q}.
\ \\

 This control problem is exactly Lemma 5.2 of \cite{RX09} with $\alpha=1/4$ therein, but we give the sketch of the proof for the completeness. Firstly, it is easy to find some $z \in D(A^{5/2})$ with $|Ay-Az| \leq \epsilon/2$, therefore it suffices to prove there exists
some control $w$ so that
\begin{equation} \label{e:ConPro}
|Au(T)-Az| \leq \epsilon/2.
\end{equation}
Secondly, decompose $u=u^\ph + u^\pl$  where $u^\ph=(I-\pi_{n_0})u$ and $u^\pl=\pi_{n_0} u$
and $n_0$ is the number in Assumption \ref{a:Q}, then equation~\eqref{e:control}
can be written as
\begin{align}
&\partial_t u^\pl + Au^\pl+ B^\pl(u,u) = 0,\label{e:loweq}\\
&\partial_t u^\ph + Au^\ph + B^\ph(u,u) = Q^\ph w.\label{e:higheq}
\end{align}
We prove \eqref{e:ConPro} in the following four steps:
\begin{enumerate}
\item \emph{Regularization of the initial data}: Let $w \equiv 0$ on $[0,T_1]$, by
some classical arguments about the regularity of Navier-Stokes equation, one has $u(T_1) \in D(A^{5/2})$,
where $T_1>0$ depends on $|Ax|$.
\item \emph{High modes lead to zero}: Choose a smooth function $\psi$ on $[T_1,T_2]$
such that $0\leq\psi\leq1$, $\psi(T_1)=1$
and $\psi(T_2)=0$, and set $u^\ph(t)=\psi(t)u^\ph(T_1)$ for $t\in[T_1,T_2]$.
Plugging this $u^\ph$ into
\eqref{e:higheq}, we obtain
$$w(t)
 =\psi'(t)(Q^\ph)^{-1}u^\ph(T_1) + \psi(t)(Q^\ph)^{-1}Au^\ph(T_1)+(Q^{\ph})^{-1}B^\ph(u(t),u(t)).
$$
\item \emph{Low modes close to $z^\pl$}: Let $u^L(t)$ be the linear interpolation
between $u^\pl(T_2)$ and $z^\pl$
for $t\in[T_2,T_3]$. Write $u(t)=\sum u_k(t)e_k$, then \eqref{e:loweq}
in Fourier coordinates is given by
\begin{equation*} 
\dot u_k + |k|^2 u_k + B_k(u,u) = 0,
\qquad k\in Z_L(N_0),
\end{equation*}
where $B_k(u,u)=B_k(u^\pl,u^\pl)+B_k(u^\pl,u^\ph)+B_k(u^\ph,u^\pl)+B_k(u^\ph,u^\ph)$. We can choose a
suitable simple $u^\ph$ to $B_k(u^\pl,u^\ph)=B_k(u^\ph,u^\pl)=0$ and make the above equation
explicitly solvable.
\item \emph{High modes close to $z^\ph$}: In the interval $[T_3,T]$
we choose $u^\ph$ as the linear interpolation
between $u^\ph(T_3)$ and $z^\ph$. By continuity, as $T-T_3$ is sufficiently small (thanks to that $T_3 \in (T_2,T)$
can be arbitrary),
$u^\pl(T)$ is still close to $z^\pl$.
\end{enumerate}
From the above four steps, we can see that $\sup_{T_1 \leq t \leq T} |A^{5/2} u(t)|<\infty$.
Moreover, since $w \equiv 0$ on $[0,T_1]$ and $\sup_{0 \leq t \leq T_1} |Au(t)|^2<\infty$,
by differentiating $|Au(t)|^2$ and applying \eqref{e:AuABu},
we have the energy inequality
\begin{equation}
|Au(T_1)|^2+\int_0^{T_1} |A^{3/2}u(s)|^2 ds \leq C \int_0^{T_1} |Au(s)|^4 ds+|Ax|^2<\infty
\end{equation}
Hence $u \in L^2([0,T], D(A^{3/2}))$. With this observation and the controllability,
we can apply Lemma 7.7 in \cite{DPD03} to obtain the conclusion (Note that our control $w$ is different from the $\bar w$ in \cite{DPD03}, this is the key point that
we can apply the argument there with $Q$ not invertible.) \\

Alternatively, with the solvability of the above control problem,
 we can apply the argument in the proof of Proposition 5.1 in \cite{RX09} to
 show irreducibility.
\end{proof}

From Proposition \ref{p:LesGam} or Proposition
\ref{p:HigAx}, $P_t$ is strong Feller. By the irreducibility, there exists a unique invariant measure $\nu$ for $X(t)$ by Doob's Theorem.  \\

Finally, let us prove the exponential mixing property \eqref{e:ExpMix}. To show this, it suffices to prove that
 \begin{equation} \label{e:ExpMixGal}
||(P^m_t)^{*} \mu-\nu_m||_{var} \leq Ce^{-ct}\left(1+\int_H |x|^2 \mu(dx)\right)
 \end{equation}
 where $c,C>0$ are independent of $m$, and $\nu_m$ is the unique measure of the approximate dynamics (see
 Lemma \ref{l:ErgDynM}). We follow exactly the coupling method in \cite{Od07} to prove
 \eqref{e:ExpMixGal},
 let us sketch out the key point as follows.  \\

 For two independent cylindrical Wiener processes $W$ and $\tilde W$, denote $X_m$ and $\tilde X_m$ the
 solutions of the equation \eqref{e:GalerkinM} driven by $W$ and $\tilde W$ respectively. For any fixed $0<T \leq 1$,
 given any two $x_1,x_2 \in D(A)$, we construct the coupling of the probabilities
 $(P^m_T)^* \delta_{x_1}$ and $(P^m_T)^* \delta_{x_2}$ as follows
 \begin{equation*}
 (V_1,V_2)=
 \begin{cases}
 (X_m(T,x_0), X_m(T,x_0)) \ \ \ {\rm if} \ x_1=x_2=x_0,
\\
(Z_1(x_1,x_2),Z_2(x_1,x_2))  \ \ {\rm if} \ x_1,x_2 \in B_{D(A)}(0,\delta) \ with \  x_1 \neq x_2,
\\
(X_m(T,x_1),\tilde X_m(T,x_2)) \ \ \ {\rm otherwise},
 \end{cases}
 \end{equation*}
where $(Z_1(x_1,x_2),Z_2(x_1,x_2))$ is the maximal coupling of
$(P^m_T)^* \delta_{x_1}$ and $(P^m_T)^* \delta_{x_2}$ (see Lemma 1.14 in \cite{Od07}) and
$B_{D(A)}(0,\delta)=\{x \in D(A); |Ax| \leq \delta\}$. It is clear
that $(V_1,V_2)$ is a coupling of $(P^m_T)^* \delta_{x_1}$ and
$(P^m_T)^* \delta_{x_2}$. We construct $(X^1,X^2)$ on $T\N$ by
induction: set $X^i(0)=x^i$ ($i=1,2$) and define
$$X^i((n+1)T)=V_i(X^1(nT),X^2(nT)) \ \ \ \ i=1,2. $$
The key point for using this coupling to show the exponential mixing
is the following lemma, which plays the same role as Lemma 2.1 in
\cite{Od07}, but we prove it by a little simpler way.
\begin{lem}
There exist some $0<T,\delta<1$ such that for any $m \in \N$, one has
a maximal coupling $(Z_1(x_1,x_2),Z_2(x_1,x_2))$ of $(P^m_T)^*
\delta_{x_1}$ and $(P^m_T)^* \delta_{x_2}$ which satisfies
\begin{equation}
\PP(Z_1(x_1,x_2)=Z_2(x_1,x_2)) \geq 3/4
\end{equation}
if $|Ax_1| \vee |Ax_2| \leq \delta$ with $\delta>0$ sufficiently
small.
\end{lem}
\begin{proof}
Since $(Z_1(x_1,x_2),Z_2(x_1,x_2))$ is maximal coupling of
$(P^m_T)^* \delta_{x_1}$ and $(P^m_T)^* \delta_{x_2}$ (see Lemma
1.14 in \cite{Od07}), one has
$$||(P^m_T)^* \delta_{x_1}-(P^m_T)^* \delta_{x_2}||_{var}=\PP\{Z_1(x_1,x_2) \neq Z_2(x_1,x_2)\}.$$
It is well known that
\begin{equation*}
\begin{split}
||(P^m_T)^* \delta_{x_1}-(P^m_T)^*
\delta_{x_2}||_{var}&=\sup_{||g||_\infty=1} \left
|\E[g(X_m(T,x_1))]-\E[g(X_m(T,x_2)] \right|
 \\
&=\sup_{||g||_\infty=1} \left |P^m_T g(x_1)-P^m_T g(x_2)\right|,
\end{split}
\end{equation*}
where $||\cdot||_\infty$ is the supremum norm.
By Proposition \ref{p:HigAx}, (noticing $||g||_0=||g||_\infty$=1 with $||\cdot||_0$ defined in section 2),
one has
\begin{equation*}
\begin{split}
 \left |P^m_T g(x_1)-P^m_T g(x_2) \right| & \leq \int_0^1
|A^{-1}D P^m_T g(\lambda x_1+(1-\lambda) x_2)| |Ax_1-Ax_2| d\lambda
 \\
 &\leq C T^{-\alpha} (1+|Ax_1|+|Ax_2|)^{2+2\alpha}
|A(x_1-x_2)| \leq 1/4
\end{split}
\end{equation*}
if choosing $\delta=T^\beta$ with $\beta>0$ sufficiently large.
Hence
$\PP(Z_1=Z_2)=1-\PP(Z_1 \neq Z_2) \geq \frac 34.$
\end{proof}
 With this lemma, one can prove the exponential mixing \eqref{e:ExpMix} by
exactly the same procedure as in \cite{Od07}.
\section{Appendix}
\subsection{Some calculus for $\tilde B_k$ and Proof of Lemma \ref{l:MixSet}} \label{sub:App1}
\begin{proof} [Some calculus for $\tilde B_k$] By $B(u,v)=\mathcal{P}[(u \cdot \nabla)v]$, we have
\begin{align*}
B(a_j cosj \xi, a_l sin l\xi)=\frac12 (l \cdot a_j) \proj[a_l cos(j+l)\xi]+\frac12 (l \cdot a_j) \proj[a_l cos(j-l)\xi],\\
B(a_j sin j \xi, a_l cos l\xi)=\frac12 (l \cdot a_j) \proj[a_l cos(j+l)\xi]-\frac12 (l \cdot a_j) \proj[a_l cos(j-l)\xi],
\end{align*}
where $\mcl P$ is the projection from $L^2(\T^3,\R^3)$ to $H$. If $j, -l \in {\Z}^3_{+}$ with $j+l \in {\Z}^3_{+}$, $\forall \ a_j \in j^{\bot}, a_l \in l^{\bot}$, we have from the above two expressions
\begin{align}
&\tilde B_{j-l} (a_j e_j, a_l e_l)=\frac{1}{2}[(l \cdot a_j) \proj_{j-l} a_l-(j \cdot a_l) \proj_{j-l} a_l],  \label{e:Bk1}\\
&\tilde B_{j+l}(a_j e_j, a_l e_l)=\frac{1}{2} [(l \cdot a_j) \proj_{j+l} a_l+(j \cdot a_l) \proj_{j+l} a_j],  \label{e:Bk4} \\
&\tilde B_k(a_j e_j, a_l e_l)=0 \ \ \ {\rm if}  \ \ k \neq j+l,j-l. \label{e:BkZero}
\end{align}
where the projection $\mathcal{P}_k: {\R}^3 \To k^{\bot}$ is defined by \eqref{e:ProK}.
For the case of $j,l \in {\Z}^3_{+}$ with $j-l \in {\
Z}^3_{-}$, we can calculate $\tilde B_{-j-l}(a_j e_j, a_l e_l)$,
$\tilde B_{j-l}(a_j e_j, a_l e_l)$ and so on by the same method.
\end{proof}
\begin{proof} [Proof of Lemma \ref{l:MixSet}]
As $k \in Z_\pl(n_0) \cap {\Z}^3_+$, for any $j, l \in {\Z}^3_{*}$ such that
\begin{align} \label{e:JLK+}
j \in Z_\ph(n_0) \cap {\Z}^3_+, \ \ l \in Z_\ph(n_0)\cap {\Z}^3_-,\ \ j \nparallel l, \ \ |j| \neq |l|, \ \ j+l=k;
\end{align}
taking an
\emph{orthogonal} basis $\{k, h_1, h_2\}$ of ${\R}^3$ where $\{h_1, h_2\}$
is an \emph{orthogonal} basis of $k^{\bot}$ with $h_1$ defined by
\begin{align*}
h_1=l \ \ \ {\rm if}  \ \ k \cdot l=0,   \ \ \ \ \ \
h_1=j-\frac{j \cdot k }{k \cdot l}l \ \ \ {\rm otherwise}.
\end{align*}
Let $p_j \in j^{\bot}, p_l \in l^{\bot}$ be represented by $p_j=a k+ b_1 h_1 +b_2 h_2$ and $p_l=\alpha k+\beta_1 h_1+\beta_2 h_2.$
 Clearly, $j,l \bot h_2$, by some basic calculation, we have
\begin{equation*} \label{e:ComputeProject}
\begin{split}
(j \cdot p_l) \mathcal{P}_k p_j+(l \cdot p_j) \mathcal{P}_k p_l=\left \{\begin{array}{cc}
-\left [\frac{(k\cdot l)(|j|^2-|l|^2)}{|j|^2|l|^2-(j \cdot l)^2} a \alpha \right] h_1+(\alpha b_2+\beta_2 a)h_2
& {\rm if} \  h_1=j-\frac{j \cdot k }{k \cdot l}l \\
-\left[\frac{j \cdot k}{j \cdot l} a \alpha \right] h_1+(\alpha b_2+\beta_2 a)h_2
& {\rm if} \  h_1=l
\end{array}
\right.
\end{split}
\end{equation*}
Since $b_2, \beta_2, a, \alpha \in {\R}$ can be arbitrarily chosen, one clearly has
\begin{equation} \label{e:choose-pj-pl}
\{j \cdot p_l) \mathcal{P}_k p_j+(l \cdot p_j) \mathcal{P}_k p_l: \ \ \  p_j \in j^{\bot},  \ p_l \in l^{\bot}\}=\ k^{\bot}.
\end{equation}
By (A3) of Assumption \ref{a:Q}, we have $rank(q_j),rank(q_l)=2$, therefore, by \eqref{e:Bk4}
and \eqref{e:choose-pj-pl},
$$\left \{\tilde B_k \left (q_{j} \ell_j e_j, q_{l} \ell_l e_l\right):
\ell_j \in j^{\bot}, \ell_l \in l^{\bot}\right\}=k^{\bot}.$$
Hence, $span \{Y_k\}=k^\bot$.
For $k \in Z_\pl(n_0) \cap {Z}^3_-$, we have the same conclusion by the same argument as above.
\end{proof}
\subsection{Proof of Lemma \ref{lem:inverse}} \label{ss:Proof of Lemma}
The key points for the proof are Proposition
\ref {prop:modified Hormander} and the following Norris' Lemma, which is
exactly Lemma 4.1 in \cite{N86}.
\begin{lem} {\bf (Norris' Lemma)} Let $a, y \in \mathbb{R}$. Let
$\beta_t, \gamma_t=(\gamma^1_t, \ldots \gamma^m_t)$
and $u_t=(u^1_t,\ldots, u^m_t)$ be adaptive processes.
Let
\begin{equation*}
a_t=a+\int_0^t \beta_s ds+\int_0^t \gamma^i_s dw^i_s, \ \ Y_t=y+\int_0^t a_s ds+\int_0^t u^i_s dw^i_s,
\end{equation*}
where $(w^1_t,\ldots,w^m_t)$ are i.i.d. standard Brownian motions.
Suppose that $T<t_0$ is a bounded stopping time such that for some
constant $C<\infty$:
\begin{equation*}
|\beta_t|, |\gamma_t|, |a_t|, |u_t| \leq C \ \ \ for \ all \ t \leq
T.
\end{equation*}
Then for any $r>8$ and $\nu>\frac{r-8}{9}$
\begin{equation*}
P\{\int_0^T Y^2_t dt<\epsilon^r, \int_0^T (|a_t|^2+|u_t|^2)dt \geq
\epsilon \}<C(t_0,q,\nu)e^{-\frac{1}{\epsilon^{\nu}}}.
\end{equation*}
\end{lem}

\begin{proof} [{Proof of Lemma \ref{lem:inverse}}] We shall drop the index $m$ of the quantities if no confusions arise. The idea of
the proof is from Theorem 4.2 of \cite{N86}, it suffices to show the
inequality in the lemma, which is equivalent to
\begin{equation} \label{e:eigenvalue2}
P\left(\inf \limits_{\eta \in \mathcal{{S}}^\pl} \Ll \mathcal{M}_t \eta, \eta \Rr \leq
\e^q \right)\leq \frac{C(p) \e^p}{t^p}  \ \ \ (\forall \ p>0)
\end{equation}
where ${\mathcal{S}}^\pl=\{\eta \in \pi^\pl H;
|\eta|=1\}$.
From (\ref{e:respresentation of Malliavin}), \eqref{e:eigenvalue2} is equivalent to
\begin{equation*}
P\left(\inf \limits_{\eta \in
\mathcal{{S}}^\pl}\sum \limits_{k \in Z_\pl(n)\setminus Z_\pl(n_0)}\sum
\limits_{i=1}^2 \int_0^t |\Ll J^{-1}_s (q^i_ke_k), \eta
\Rr|^2ds \leq \e^q \right) \leq \frac{C(p)\e^p}{t^p},
\end{equation*}
(recall $q^i_k$ is the i-th column vector of the matrix $q_k$, see Assumption \ref{a:Q}),
which is implied by
\begin{equation}  \label{e:split}
D_{\theta} \sup_j \sup \limits_{\eta \in \mathcal{D}_j}P
\left(\int_0^t \sum \limits_{k \in Z_\pl(n) \setminus Z_\pl(n_0)} \sum
\limits_{i=1}^2|\Ll J^{-1}_s (q^i_ke_k), \eta
\Rr|^2ds \leq \e^q \right) \leq \frac{C(p)\e^p}{t^p}
\end{equation}
where $\{\mathcal{D}_j\}_j$ is a \emph{finite} $\theta$-radius disk cover of ${\mathcal{S}}^\pl$
(due to the compactness of ${\mathcal{S}}^\pl$) and $D_{\theta}=\#
\{\mathcal{D}_j\}$. Define a stopping time $\tau$ by
\begin{equation} \label{e:StoTim}
\tau=\inf \{s>0;\left |\ek(s)J^{-1}_s-Id \right|_{\mathcal{L}(H)}>c\}.
\end{equation}
where $c>0$ is a sufficiently small but fixed number. It is easy to see
that (\ref{e:split}) holds as long as for any $\eta \in \mathcal{S}^\pl$,
we have some neighborhood $\mathcal{N}(\eta)$ of $\eta$ and some
$k \in Z_\pl(n) \setminus Z_\pl(n_0)$, $i \in \{1,2\}$ so that
\begin{equation} \label{e:qiksmall}
\sup \limits_{\eta^{'} \in \mathcal{N}(\eta)}P \left(\int_0^{t \wedge \tau}
|\Ll J^{-1}_s (q^i_k e_k), \eta^{'}
\Rr|^2ds \leq \e^q \right) \leq \frac{C(p)\e^p}{t^p} \ \ (\forall
\ p>0).
\end{equation}
The above argument is according to \cite{N86} (see Claim 1 of the proof of Theorem 4.2), one may see the greater details there.
\ \\

Let us prove \eqref{e:qiksmall}. According to the restricted H$\ddot{o}$rmander condition and Definition
\ref{defn:Hormander systems}, for any $\eta \in \mathcal{S}^\pl$, there exists a $K \in {\bf K}$
satisfying for all $y \in \pi_m H$
$$|\Ll K(y), \eta \Rr|^2 \geq \delta |\eta|^2$$
where $\delta>0$ is a constant independent of $y$.
Without loss of generality, assume that
$K \in {\bf K}_2$, so there exists some $q^i_k e_k$ and $q^j_l e_l$ such that
$$K_0:=q^i_k e_k, \ K_1:=[A_m y+B^{\pl}_m(y,y),q^i_k e_k], \ K=K_2:=[q^j_l e_l,K_1].$$
Take
\begin{align*}
Y(t)=\Ll J^{-1}_t K_1(X(t)), \eta \Rr, \ \ u^i(t)=0, \ \
a(t)=\Ll J^{-1}_t K_2(X(t)),\eta \Rr,
\end{align*}
applying Norris lemma with $t_0=1$ therein, we have
\begin{equation*}  
P\left(\int_0^{t \wedge \tau}|\Ll J^{-1}_s K_1(X(s)), \eta \Rr|^2ds \leq \e^r,
\int_0^{t \wedge \tau}|\Ll J^{-1}_s K_2(X(s)), \eta \Rr|^2 ds \geq \e
\right) \leq C(p,\nu)e^{-\frac{1}{\e^{\nu}}}
\end{equation*}
On the other hand, by \eqref{e:JacEst2}, \eqref{e:StoTim} and Chebyshev's inequality, it is easy to have
\begin{equation*}
\begin{split}
& \ \ P\left(\int_0^{t \wedge \tau}|\Ll J^{-1}_s K_2(X(s)), \eta \Rr|^2 ds \leq \e
\right) \\
&=P\left(\int_0^{t \wedge \tau}\frac 1{\ek(s)^2}|\Ll \ek(s)J^{-1}_s K_2(X(s)), \eta \Rr|^2 ds \leq \e
\right) \\
&\leq P \left(\int_0^{t \wedge \tau} |\Ll \ek(s)J^{-1}_s K_2(X(s)), \eta \Rr|^2 ds \leq \e
\right) \leq P\left(\tau \leq \frac{c\e}{t}\right) \leq \frac{C(p) \e^p}{t^p}.
\end{split}
\end{equation*}
Hence,
\begin{equation*}
P\left(\int_0^{t \wedge \tau}|\Ll J^{-1}_s K_1(X(s)), \eta \Rr|^2ds \leq \e^r
\right) \leq \frac{C(p) \e^p}{t^p}.
\end{equation*}
By a similar but simpler arguments, (recalling $K_0=q^i_k e_k$), we have
\begin{equation}  \label{e:Aim}
P\left(\int_0^{t \wedge \tau}|\Ll J^{-1}_s (q^i_k e_k), \eta \Rr|^2ds \leq \e^{r^2}
\right) \leq \frac{C(p) \e^p}{t^p}
\end{equation}
for all $p>0$. \\

Hence, for any $\eta \in
\mathcal{S}^{\pl}$, we have some $q^i_k e_k$ satisfying
\eqref{e:Aim}. Take the neighborhood $\mathcal{N}(\eta)$ small enough and $q=r^2$,
 by the continuity, we have \eqref{e:qiksmall} immediately.
\end{proof}
\subsection{Proof of some technical lemmas} \label{app:TecLem}
In this subsection, we need a key estimate as follows (see Lemma D.2 in \cite{FR08}): For any
$\gamma>1/4$ with $\gamma \neq 3/4$, we have
\begin{equation} \label{e:NonLinEst}
|A^{\gamma-1/2} B(u,u)| \leq C(\gamma) |A^\gamma u|^2 \ \ {\rm for \ any} \ u \in D(A^{\gamma}).
\end{equation}
By \eqref{e:NonLinEst}
and Young's inequality, we have
\begin{equation} \label{e:AuABu}
\begin{split}
|\Ll A^{\gamma} u, A^{\gamma} B(u,v) \Rr| & \leq |A^{\gamma+1/2} u| |A^{\gamma-1/2} B(u,v)| \\
& \leq |A^{\gamma+1/2} u|^2+C(\gamma)|A^{\gamma}u|^2|A^{\gamma} v|^2
\end{split}
\end{equation}
\begin{proof} [{Proof of Lemma \ref{l:HighDerivative}}]
We shall drop the index $m$ of quantities if no confusions arise. By It$\hat{o}$ formula, we have
\begin{equation}
\begin{split}
& d\left [|A^\gamma D_hX(t)|^2 \ek(t) \right]+2|A^{1/2+\gamma}D_hX(t)|^2 \ek(t) \\
&+2\Ll A^{\gamma}D_hX(t),A^{\gamma}\tilde B \left [D_hX(t),X(t) \right ] \Rr \ek(t) dt \\
&+K |A^\gamma D_{h}X(t)|^2 |AX(t)|^2 \ek(t) dt=0
\end{split}
\end{equation}
where $\tilde B \left (D_{h}X(t),X(t) \right)=B \left (D_{h}X(t),X(t) \right )
+B \left (X(t),D_{h}X(t) \right )$.
Thus, one has by \eqref{e:AuABu}
\begin{equation} \label{e:Energy}
\begin{split}
& \ \ |A^\gamma D_{h}X(t)|^2 \ek(t)+\int_0^t |A^{\gamma+1/2}D_{h}X(s)|^2 \ek(s) ds \\
& \leq |A^\gamma h|^2+\int_0^t C|A^\gamma D_{h}X(s)|^2
|A^\gamma X(s)|^2\ek(s)ds \\
& \ \ -K \int_0^t |A^\gamma D_{h}X(s)|^2  |AX(s)|^2 \ek(s) ds
\end{split}
\end{equation}
By Poincare inequality, we have $|Ax| \geq |A^{\gamma} x|$, and therefore as $K \geq C$,
\begin{equation*}
|A^\gamma D_{h}X(t)|^2 \ek(t)+\int_0^t |A^{\gamma+1/2}D_{h}X(s)|^2 \ek(s) ds \leq |A^\gamma h|^2.
\end{equation*}
As to \eqref{e:DhHXL} and \eqref{e:DhLXH}, we only prove \eqref{e:DhHXL}, similarly for the other.
By an estimate similar to
\eqref{e:Energy},  \eqref{e:DhX} and \eqref{e:IntEmk} (noticing $D_{h^\ph}X^\pl(0)=0$), we have
\begin{equation}
\begin{split}
& \ \ \  |A^\gamma D_{h^\ph}X^\pl(t)|^2 \ek(t)+\int_0^t
|A^{\gamma+1/2} D_{h^\ph} X^\pl (s)|^2 ds \leq  \\
& \leq \int_0^t [C|A^\gamma D_{h^\ph}X(s)|^2 |A^\gamma X(s)|^2
-K|A^\gamma D_{h^\ph} X^\pl(s)|^2 |AX(s)|^2] \ek(s)ds \\
& \leq C \int_0^t \left [|A^\gamma D_{h^\ph}X(s)|^2 \mathcal{E}_{\frac K2}(s)\right]
\left[|AX(s)|^2 \mathcal{E}_{\frac K2}(s)\right] ds \\
& {\leq} C |A^\gamma h|^2 \int_0^t |AX(s)|^2
\mathcal{E}_{\frac K2}(s) ds \leq \frac {2C}K |A^\gamma h|^2.
\end{split}
\end{equation}
 As to \eqref{e:ArDhX}, by the classical interpolation inequality
$$|A^r D_hX(s)|^2 \leq |A^\gamma D_hX(s)|^{2(1-2(r-\gamma))}|A^{1/2+\gamma} D_h X(s)|^{4(r-\gamma)},$$
H$\ddot{o}$lder's inequality and \eqref{e:DhX}, we have
\begin{equation*}
\begin{split}
\int_0^t |A^r D_hX(s)|^2 \ek(s) ds & \leq
[\int_0^t |A^{\gamma+1/2}D_hX(s)|^2 \mathcal{E}_{\frac K2}(s)ds]^{2(r-\gamma)}
[\int_0^t |A^\gamma D_hX(s)|^2\mathcal{E}_{\frac K2}(s)ds]^{1-2(r-\gamma)} \\
& {\leq}C t^{1-2(r-\gamma)} |A^\gamma h|^{2}.
\end{split}
\end{equation*}
\eqref{e:EkMat} immediately follows from applying It$\hat{o}$ formula to $|\ek(t) \int_0^t \Ll v,dW_s \Rr|^2$.
\end{proof}

\begin{proof} [{Proof of Lemma \ref{l:JacEst}}]
By \eqref{e:Jt} and the evolution equation governing $D_h X^\pl$, using the same method as proving \eqref{e:DhX}, we immediately
have \eqref{e:JacEst1} and \eqref{e:DhXpl}. Recall that $J_t$ and $J^{-1}_t$ are both the dynamics in
$\pi^{\pl} H$, thus
the operator $J_t$ is bounded invertible. Let $C$ be some constant only depends on $n$ (see \eqref{e:LowHigH}), whose values can vary from line to line. By the fact $|A|_{\mcl L(\pi^\pl H)} \leq C$
and \eqref{e:J-1t}, for any $h \in \pi^\pl H$, we have by differentiating $|J^{-1}_t h|^2 \ek(t)$
\begin{equation} \label{e:EnergyJ-1}
\begin{split}
 & \ \ \ |J^{-1}_t h|^2 \ek(t)+K\int_0^t |J^{-1}_sh|^2 |AX(s)|^2 \ek(s) ds \\
 &\leq |h|^2+2
\int_0^t |J^{-1}_s h| |J^{-1}_s Ah| \ek(s) ds \\
& \ \ +C
\int_0^t |J^{-1}_s h| |J^{-1}_s A^{-\frac12}|_{\mathcal{L}(\pi^{\pl}H)}
\cdot |A^{1/2}B^\pl(h,X(s))| \ek(s)ds \\
& \leq |h|^2+C\int_0^t |J^{-1}_s|_{\mathcal{L}(H)}^2 |h|^2
\ek(s) ds+C\int_0^t |J^{-1}_s|_{\mathcal{L}(H)}^2 |AX(s)|
|h|^2 \ek(s) ds,
\end{split}
\end{equation}
where the last inequality is by \eqref{e:NonLinEst}.
Hence,
$$ |J^{-1}_t|^2 \ek(t)+K\int_0^t |J^{-1}_s|_{\mathcal{L}(H)}^2 |AX(s)|^2 \ek(s) ds
\leq 1+C\int_0^t |J^{-1}_s|_{\mathcal{L}(H)}^2(1+|AX(s)|^2)
\ek(s) ds,$$
as $K$ is sufficiently large, we have
$|J^{-1}_t|_{\mathcal{L}(H)}^2 \ek(t) \leq 1+C\int_0^t
|J^{-1}_s|_{\mathcal{L}(H)}^2 \ek(s) ds,$
which immediately implies \eqref{e:J-1}.
\ \\

To prove \eqref{e:JacEst2}, by \eqref{e:J-1t}, we have
\begin{equation}
\begin{split}
\left |\ek(t) J^{-1}_t h-h \right| &\leq
\int_0^t |J^{-1}_s A h| \ek(s) ds+\int_0^t |J^{-1}_sA^{-1/2}|_{\mathcal{L}(\pi^\pl H)}
|A^{1/2} B^\pl(h,X(s))| \ek(s) ds  \\
& \leq C\int_0^t |J^{-1}_s|_{\mathcal{L}(H)}|h| \ek(s)ds+C\int_0^t
|J^{-1}_s|_{\mathcal{L}(H)}|h| |AX(s)| \ek(s) ds,
\end{split}
\end{equation}
thus, by \eqref{e:J-1} and \eqref{e:IntEmk},
\begin{equation*}  
\begin{split}
\left |\ek(t) J^{-1}_t-Id \right|_{\mathcal{L}(H)}
& \leq C\int_0^t \ek(s) |J^{-1}_s|_{\mathcal{L}(H)}ds+
\int_0^t  \ek(s) |J^{-1}_s|_{\mathcal{L}(H)} |AX(s)|ds \\
& \leq Ct^{\frac 12} \left[\int_0^t \mathcal{E}_{K}(s) |J^{-1}_s|^2_{\mathcal{L}(H)}ds\right]^{\frac 12}
+t^{\frac 12}  \left[\int_0^t \ek(s) |J^{-1}_s|^2_{\mathcal{L}(H)} \ \ek(s) |AX(s)|^2 ds\right]^{\frac 12}\\
& \leq t^{1/2} Ce^{Ct}
\end{split}
\end{equation*}
where the last inequality is due to \eqref{e:J-1}.
As for \eqref{e:JmQlEmk}, by Parseval's identity and \eqref{e:J-1},
\begin{equation}
\begin{split}
& \ \ \E \left[\int_0^t \ek^2(s) |(J^{-1}_sQ^\pl)^{*} h|^2 ds \right]
=\sum_{k \in Z_\pl (n)}\sum_{i=1}^2 \E \left[\int_0^t \mathcal{E}_{2K}(s) |\Ll J^{-1}_s(q^i_k e_k),h \Rr|^2 ds
\right] \\
&\leq \sum_{k \in Z_\pl(n)}\sum_{i=1}^2 \E \left[\int_0^t \mathcal{E}_{2K}(s) |J^{-1}_s(q^i_k e_k)|^2 ds
\right] |h|^2
\leq tCe^{Ct} \sum_{k \in Z_\pl
(n)\setminus Z_\pl(n_0)}\sum_{i=1}^2|q^i_k e_k|^2 |h|^2.
\end{split}
\end{equation}
By an estimate similar to \eqref{e:Energy}, we have
\begin{equation*} 
\begin{split}
 & \ \ \ |A\D_v X^\pl (t)|^2 \ek(t)+\int_0^t |A^{3/2} \D_v X^\pl (t)|^2 \ek(s) ds \\
 &\leq \frac 12 \int_0^t |A \D_v X^\pl (s)|^2 \ek(s)ds +\frac 12 \int_0^t |AQ^\pl v^\pl(s)|^2
 \ek(s)ds \\
& \leq
 \frac 12 \int_0^t |A \D_v X^\pl (s)|^2 \ek(s)ds +C \int_0^t |v^\pl(s)|^2
 \ek(s)ds
\end{split}
\end{equation*}
which implies \eqref{e:Mal1OrdEst} by Gronwall's inequality.

As to \eqref{e:Mix2OrdEst}, write down the differential equation for $\D_v D_h X^\pl (t)$, and apply It$\hat{o}$
formula, we have
\begin{equation*} \label{e:EnergyDvDhX}
\begin{split}
 & \ \ \ \ |\D_v D_h X^\pl (t)|^2 \ek(t)+2\int_0^t |A^{1/2} \D_v D_h X^\pl (t)|^2 \ek(s) ds \\
 &\leq \int_0^t |\D_v D_h X^\pl (s)| \left(|\tilde B^\pl (\D_v D_h X^\pl (s),X(s))|+|\tilde B^\pl (D_h X^\pl (t),\D_v X(s))| \right) \ek(s)ds \\
 &\ \ -K \int_0^t |\D_v D_h X^\pl (s)|^2 \ek(s) ds.
\end{split}
\end{equation*}
By \eqref{e:NonLinEst} and Young's inequality,
\begin{equation*} \label{e:EnergyDvDhX}
\begin{split}
 & \ \ \ \ |\D_v D_h X^\pl (t)|^2 \ek(t)+\int_0^t |A^{1/2} \D_v D_h X^\pl (t)|^2 \ek(s) ds \\
 &\leq
\int_0^t |\D_vD_hX^\pl (s)|^2 \left(|A^{\frac 12}X(s)|^2+1\right) \ek(s) ds
+\int_0^t |A^{\frac 12} \D_v X^\pl (s)|^2 |A^{\frac 12} D_h X^\pl (s)|^2 \ek(s)ds \\
& \ \ \ \
-K\int_0^t |\D_vD_h X^\pl (s)|^2|AX (s)|^2 \ek(s) ds,
\end{split}
\end{equation*}
as $K$ is sufficiently large, by Poincare inequality, $|A|_{\mcl L(\pi^\pl H)} \leq C$,
\eqref{e:Mal1OrdEst} and \eqref{e:DhX},
we have
\begin{equation*}
\begin{split}
& \ \ |\D_v D_h X^\pl (t)|^2 \ek(t)+\int_0^t |A^{1/2} \D_v D_h X^\pl (t)|^2 \ek(s) ds \\
 &\leq
\int_0^t |\D_vD_hX^\pl (s)|^2 \ek(s) ds
+C \int_0^t |\D_v X^\pl (s)|^2 \mcl E_{K/2}(s) |D_h X^\pl (s)|^2 \mcl E_{K/2}(s)ds \\
& \leq \int_0^t |\D_vD_hX^\pl (s)|^2 \ek(s) ds+C|h|^2\int_0^t e^{t-s} |v^\pl(s)|^2  \mcl E_{K/2}(s)ds.
\end{split}
\end{equation*}
By Gronwall's inequality, we obtain \eqref{e:Mix2OrdEst} immediately.
Similarly, \eqref{e:Mal2OrdEst} can be obtained by
\begin{equation*}
\begin{split}
& \ \ \ |\D^2_{v_1v_2} X^\pl (t)|^2 \ek(t) \\
& \leq C \int_0^t [|A \D_{v_1} X^\pl (s)|^2 \mathcal{E}_{K/2} (s)] \ [|AD_{v_2} X^\pl (s)|^2
\mathcal{E}_{K/2} (s)]ds \\
& \leq t Ce^{Ct}\left[\int_0^t |v_1^\pl(s)|^2
\mathcal{E}_{K/2}(s)ds\right] \left[\int_0^t |v_2^\pl(s)|^2 \mathcal{E}_{K/2}(s)ds\right].
\end{split}
\end{equation*}
where the last inequality is due to \eqref{e:Mal1OrdEst}.
\end{proof}


\bibliographystyle{amsplain}

\end{document}